\documentclass[a4paper,12pt,reqno]{amsart}
\usepackage{amssymb, amsmath, amsfonts, amscd}
\usepackage{mathrsfs, latexsym}
\usepackage[all,ps,cmtip]{xy}
\usepackage{array, longtable}
\usepackage{bm, enumitem}
\usepackage{xcolor}
\title{On minimal varieties growing from quasismooth weighted hypersurfaces}
\date{\today}

\author{Meng Chen}
\address{\rm School of Mathematical Sciences, Fudan University, Shanghai 200433, China}
\email{mchen@fudan.edu.cn}

\author{Chen Jiang}
\address{\rm Shanghai Center for Mathematical Sciences, Fudan University, 
Shanghai 200438, China}
\email{chenjiang@fudan.edu.cn}

\author{Binru Li}
\address{\rm School of Mathematics, Sichuan University, Chengdu 610064, China}
\email{binruli@scu.edu.cn}

\thanks{This project was supported by NSFC for Innovative Research Groups (\#12121001) and National Key Research and Development Program of China (\#2020YFA0713200). The first author was supported by National Natural Science Foundation of China (\#12071078,\#11731004). The second author was supported by Start-up Grant No.~SXH1414010. }
\newcommand{\bC}{{\mathbb C}}
\newcommand{\bQ}{{\mathbb Q}}
\newcommand{\bP}{{\mathbb P}}
\newcommand{\roundup}[1]{\lceil{#1}\rceil}
\newcommand{\rounddown}[1]{\lfloor{#1}\rfloor}

\newcommand\Vol{\text{\rm Vol}}
\newcommand\lrw{\longrightarrow}

\newcommand\OO{{\mathcal{O}}}

\newcommand\bZ{{\mathbb{Z}}}

\newcommand{\lsgeq}{\succcurlyeq}

\newcommand{\cone}{\text{\rm cone}}
\newcommand{\modr}{\text{\rm mod}}

\newtheorem{thm}{Theorem}[section]
\newtheorem{lem}[thm]{Lemma}

\newtheorem{prop}[thm]{Proposition}

\theoremstyle{definition}
\newtheorem{defn}[thm]{Definition}

\newtheorem{question}[thm]{Question}
\newtheorem{exmp}[thm]{Example}

\newtheorem{conj}[thm]{Conjecture}

\newtheorem{rem}[thm]{Remark}

\theoremstyle{remark}

\begin{document}
\begin{abstract} 
This paper concerns the construction of minimal varieties with small canonical volumes. The first part devotes to establishing an effective nefness criterion for the canonical divisor of a weighted blow-up over a weighted hypersurface, from which we construct plenty of new minimal $3$-folds including $59$ families of minimal $3$-folds of general type, several infinite series of minimal $3$-folds of Kodaira dimension $2$, $2$ families of minimal $3$-folds of general type on the Noether line, and $12$ families of minimal $3$-folds of general type near the Noether line. In the second part, we prove effective lower bounds of canonical volumes of minimal $n$-folds of general type with canonical dimension $n-1$ or $n-2$. Examples are provided to show that the theoretical lower bounds are optimal in dimension at most $5$ and nearly optimal in higher dimensions. 
\end{abstract}
\maketitle

%%%%%%%%%
\pagestyle{myheadings}
\markboth{\hfill M. Chen, C. Jiang, B. Li\hfill}{\hfill Minimal varieties growing from weighted hypersurfaces\hfill}
\numberwithin{equation}{section}
%%%%%%%%%%%%
%\tableofcontents

\section{Introduction}

In birational geometry, the minimal model program (in short, MMP) predicts that any projective variety is birationally equivalent to a minimal variety (i.e., a normal projective variety with a nef canonical divisor $K$ and with at worst $\bQ$-factorial terminal singularities) or a Mori fiber space. Though there still remain some challenging open problems such as the abundance conjecture, the MMP theory is very successful (see, e.g., \cite{KMM, K-M, BCHM}). On the other hand, from the point of view of the classification theory, the equally important question might be seeking out a concrete minimal model of the given variety and to calculate its birational invariants such as the canonical volume, plurigenus, holomorphic Euler characteristic and so on. 

The pioneering work of Reid in 1979 (\cite{Reid79}) gave the famous list of 95 weighted hypersurface K3 surfaces. Following Reid's strategy, Iano-Fletcher \cite{Fle00} provided several lists of weighted complete intersection $3$-folds which are minimal and have at worst terminal cyclic quotient singularities. A remarkable example from Iano-Fletcher's lists is the general weighted hypersurface of degree $46$, say $X_{46}\subset \bP(4,5,6,7,23)$, which has canonical volume as small as $\frac{1}{420}$ and canonical stability index $r_s(X_{46})=27$ and so, in particular, the $26$-canonical map of $X_{46}$ is non-birational. Among all known examples of $3$-folds of general type, this example has the smallest canonical volume and the largest canonical stability index,
and it may account for the complexity of higher dimensional birational geometry. In fact, intentionally constructing minimal varieties is an important but quite difficult problem. This kind of efforts appear in some established work such as \cite{JK01, CR02, Rei05, BKR, BR1, B-K, BR2} and so on. See also recent works for constructing examples \cite{TW, ETW}.

Another motivation for constructing minimal varieties, especially with small canonical volumes, is the following open question raised by Hacon and McKernan (see \cite[Problem 1.5]{H-M06}):

\begin{question}\label{P1} For any $n\geq 3$, find the optimal constant $r_n\in \bZ_{>0}$ such that, for any projective $n$-fold of general type, the $m$-canonical map is birational for all $m\geq r_n$.
\end{question}

It is known that Question~\ref{P1} (see also \cite[Problem 2]{EXPI}) is essentially equivalent to the following:

\begin{question}\label{P2} For any $n\geq 3$, find the optimal constant $v_n\in \bQ_{>0}$ such that, for any projective $n$-fold $Y$ of general type, the canonical volume $\Vol(Y)\geq v_n$. 
\end{question}

A naive way of constructing desired varieties could be, starting from a singular hypersurface in $\bP^n$, to obtain a smooth model by resolving singularities and to calculate their birational invariants. Unlike the surface case, this would be very difficult in higher dimensions as MMP and birational morphisms are more complicated than those in the surface case. 
%The reason is that the there isn't any analogy of the Castelnuovo theorem%: a smooth projective surface is relatively minimal if and only if the surface has no $(-1)$-curves. 

In this paper, we consider a special construction starting from a well-formed quasismooth weighted hypersurface $X$ (see Subsection~\ref{subsec wps} for definitions), with only isolated singularities, which has exactly one non-canonical singularity $P$. One takes a weighted blow-up along the center $P\in X$ to obtain a higher model $Y$. %, which has better singularities than $X$ does. 
To make sure that $Y$ is almost minimal, as the key observation, we have the following nefness criterion:
% where, for the definition of ``well-formed'', one may refer to Definition~\ref{wellform}:

\begin{thm}\label{nefness}
Let $X=X^n_d\subset \mathbb{P}(b_1, \dots, b_{n+2})$ be an $n$-dimensional well-formed quasismooth general hypersurface of degree $d$ with $\alpha=d-\sum_{i=1}^{n+2}b_i>0$ where $b_1, \dots, b_{n+2}$ are not necessarily sorted by size. Denote by $x_1,\dots,x_{n+2}$ the homogenous coordinates of $\mathbb{P}(b_1, \dots, b_{n+2})$.
Denote by $\ell$ the line $(x_1=x_2=\dots =x_{n}=0)$ in $\mathbb{P}(b_1, \dots, b_{n+2})$. Suppose that $X\cap\ell$ consists of finitely many points and take $Q\in X\cap \ell$.
Assume that $X$ has a cyclic quotient singularity of type $\frac{1}{r}(e_1,\dots, e_n)$ at $Q$ where $e_1, \dots, e_n>0$, $\gcd(e_1, \dots, e_n)=1$, $\sum_{i=1}^ne_i<r$ and that $x_1,\dots,x_n$ are also the local coordinates of $Q$ corresponding to the weights $\frac{e_1}{r},\dots,\frac{e_n}{r}$ respectively. Let $\pi: Y\to X$ be the weighted blow-up at $Q$ with weight $(e_1,\dots, e_n)$.
Suppose that there exists an index $k\in \{1,\dots,n\}$ such that 
\begin{enumerate}
\item $\alpha e_{j} \geq b_{j}(r-\sum_{i=1}^ne_i)$ for each $j\in \{1, \dots, \hat{k},\dots, n\}$;
\item $\alpha dre_{k}\geq b_{k}b_{n+1}b_{n+2}(r-\sum_{i=1}^ne_i);$
\item a general hypersurface of degree $d$ in $\mathbb{P}(b_{k}, b_{n+1}, b_{n+2})$ is irreducible;
\item $\mathbb{P}(e_1,\dots,e_n)$ is well-formed.
\end{enumerate}
Then $K_Y$ is nef and $\nu(Y)\geq n-1$ where $\nu(Y)$ denotes the numerical Kodaira dimension. \end{thm}

Under the setting of Theorem~\ref{nefness}, if moreover $Y$ has at worst canonical singularities, then $Y$ will correspond to a desired minimal variety of Kodaira dimension at least $n-1$, which is the case at least in dimension $3$. In Section~\ref{sec 4}, we provide $59$ families of concrete minimal $3$-folds of general type in Table~\ref{tableA}, Table~\ref{tableAp}, Table~\ref{tableC}, and Table~\ref{tableC+}. All of these examples are different from those found by Iano-Fletcher as they have Picard numbers at least $2$. Moreover most of our examples have different deformation invariants from those of known ones. Here we mention several very interesting minimal $3$-folds in our tables:
\begin{itemize}
\item[$\diamond$] The minimal models of both the general hypersurface of degree $13$ in $\bP(1,1,2,3,5)$ (see Table~\ref{tableA}, No.~2) and the general hypersurface of degree $15$ in $\bP(1,1,2,3,7)$ (see Table~\ref{tableA}, No.~3) have $p_g=2$ and $K^3=\frac{1}{3}$. These are new examples attaining minimal volumes among minimal $3$-folds of general type with $p_g\geq 2$, and they justify the sharpness of \cite[Theorem 1.4]{Chen07}. 
\item[$\diamond$] The minimal model of the general hypersurface of degree $40$ in 
$\bP(1,1,5,8,20)$ (see Table~\ref{tableC}, No.~7) is a smooth minimal $3$-fold with $p_g=7$ and $K^3=6$; the minimal model of the general hypersurface of degree $120$ in 
$\bP(1,1,17,24,60)$ (see Table~\ref{tableC}, No.~11) is a smooth minimal $3$-fold with $p_g=19$ and $K^3=22$. Both examples lie on the Noether line
$K^3=\frac{4}{3}p_g-\frac{10}{3}$. These are new examples which are not birationally equivalent to those constructed by Kobayashi \cite{Kob} and by Chen and Hu \cite{C-H} (see Remark~\ref{remark noether line}). 
\item[$\diamond$] The minimal model of the general hypersurface of degree $70$ in $\bP(1,1,10,14,35)$ (see Table~\ref{tableC}, No.~10) has $p_g=10$ and $K^3=\frac{301}{30}$, which lies very closely above the Noether line. One may refer to \cite{Noether, Noether_Add} for the importance of this example. 
\end{itemize}

Another interesting application of Theorem~\ref{nefness} is that one may find infinite series of families of minimal $3$-folds of Kodaira dimension $2$ in Tables~\ref{tab kod 2 6r}$\sim$\ref{tab kod 2 4r+6} (see Table~\ref{tableB} in Appendix~\ref{appendix} for more concrete examples). As far as we know, there are very few known examples of minimal $3$-folds of Kodaira dimension $2$ appearing in literature. These examples may be valuable in future study as they are useful in the classification theory of $3$-folds. It is also possible to provide a lot of higher dimensional examples. In Proposition~\ref{HD} and Lemma~\ref{dim+} we provide methods to construct higher dimensional minimal varieties starting from lower dimensions. 
%At least one may do as what Proposition~\ref{HD} and Lemma~\ref{dim+} hint. 

In the second part of this paper, we mainly study Question~\ref{P2} for minimal varieties with high canonical dimensions. For a given minimal variety $Y$ of general type, the {\it canonical dimension} is defined as $\text{can.dim}(Y)=\dim \overline{\Phi_{|K_Y|}(Y)}$. 
The extremal case with $\text{can.dim}(Y)=\dim(Y)=n$ was first studied by Kobayashi \cite{Kob}, in which case the following optimal inequality holds:
$$K_Y^n\geq 2p_g(Y)-2n. $$
The above equality may hold when $Y$ is certain double cover over $\bP^n$. Hence it is natural to study those cases with $\text{can.dim}(Y)<n$. 

Our main theorems are as follows:
 
\begin{thm}\label{thm(n-1)} Let $Y$ be a minimal projective $n$-fold of general type with canonical dimension $n-1$ ($n\geq 3$). Then %$K_Y^n\geq \frac{2}{n-1}$. 
%$$K_Y^n\geq \max\{\frac{2(p_g(Y)-n+1)}{n-1}, \frac{1}{(n-1)^2}\roundup{\frac{8}{3}\big((n-1)(p_g(Y)-n+1)-1\big)}\}.$$
%In particular, 
$$K_Y^n\geq \begin{cases}
\frac{2}{n-1}, &\text{ for }3\leq n\leq 5;\\
\frac{1}{(n-1)^2}\lceil \frac{8(n-2)}{3} \rceil, &\text{ for }n\geq 6.
\end{cases}$$ 
\end{thm} 

\begin{thm}\label{thm(n-2)} Let $Y$ be a minimal projective $n$-fold of general type with canonical dimension $n-2$ ($n\geq 3$). Then 
$$K_Y^n\geq \begin{cases}
\frac{1}{3}, &\textrm{for } n=3;\\
\frac{1}{(n-1)(n-2)}, &\textrm{for } 4\leq n\leq 11;\\
	\frac{4n-14}{3(n-2)^3}, &\textrm{for } n\geq 12.
	\end{cases}$$
\end{thm} 

In fact, we prove slightly general inequalities in Theorem~\ref{v(n-1)} and Theorem~\ref{v(n-2)}. It is interesting to know whether the lower bounds in Theorem~\ref{thm(n-1)} and Theorem~\ref{thm(n-2)} are close to the optima. 

Thanks to the effective construction in the first part of this paper, we are able to find supporting examples which show that, at least, both Theorem~\ref{thm(n-1)} and Theorem~\ref{thm(n-2)} are optimal for $n\leq 5$. We also provide higher dimensional examples of which the canonical volumes are nearly optimal. 

The structure of this article is as follows. In Section~\ref{sec 2}, we collect basic notions and preliminary results. In Section~\ref{sec 3}, we prove the nefness criterion (i.e., Theorem~\ref{nefness}), which enables us to tell when a weighted blow-up of a well-formed quasismooth weighted hypersurface induces a minimal model with canonical singularities. Section~\ref{sec 4} devotes to presenting concrete examples of minimal $3$-folds with Kodaira dimension at least $2$. In Section~\ref{sec 5}, we mainly study the lower bound of canonical volumes of minimal varieties of general type with higher canonical dimensions. Thanks to our established construction, we manage to find many series of supporting examples in all dimensions, in Section~\ref{sec 6}, which justify the inequalities in Theorem~\ref{thm(n-1)} and Theorem~\ref{thm(n-2)}. {}Finally we put a list of $46$ families of concrete minimal $3$-folds of Kodaira dimension $2$ in Appendix~\ref{appendix}. %These examples are useful for our future study. 

\section{Preliminaries}\label{sec 2}
Throughout we work over any algebraically closed field of characteristic $0$.
%For a real number $a$, we denote $\lfloor a\rfloor$ to be the integral part, which is the integer defined by 
% {$a-1<\lfloor a\rfloor \leq a$},
%and we denote the fractional part $\{a\}=a-\lfloor a\rfloor$.

\subsection{Canonical volume and canonical dimension}\

Let $Z$ be a smooth projective variety. The {\it canonical volume} of $Z$ is defined as
$$\Vol(Z)=\lim_{m\to \infty}\frac{\dim(Z)!\ h^0(Z,mK_Z)}{m^{\dim(Z)}}.$$
For an arbitrary normal projective variety $X$, the {\it geometric genus} of $X$ is defined as $p_g(X)=p_g(Z)=h^0(Z, K_Z)$, and the {\it canonical volume} of $X$ is defined as 
$\Vol(X) = \Vol(Z),$
where $Z$ is a smooth birational model of $X$. It is known that both $p_g(X)$ and $\Vol(X)$ are independent of the choice of $Z$. Moreover, if $X$ has at worst canonical singularities and $K_X$ is nef, then $\Vol(X)=K^{\dim(X)}_X.$ 

A normal projective variety $X$ is {\it of general type} if $\Vol(X)>0$. A projective variety $X$ is {\it minimal} if $X$ is normal $\mathbb{Q}$-factorial with at worst terminal singularities and $K_X$ is nef.

\begin{defn} For a normal projective variety $X$ with at worst canonical singularities, define the {\it canonical dimension} of $X$ as
$$\text{\rm can.dim}(X)=\begin{cases}
-\infty, & \text{if } p_g(X)=0;\\
0, & \text{if } p_g(X)=1;\\
\dim\overline{\Phi_{|K_X|}(X)}, &\textrm{otherwise}.\end{cases}$$
Here $\Phi_{|K_X|}$ is the rational map defined by the linear system $|K_X|$, and $\overline{\Phi_{|K_X|}(X)}$ is called the {\it canonical image} of $X$. 
\end{defn}

\subsection{Kodaira dimension and numerical Kodaira dimension}\

\begin{defn}
Let $D$ be a $\mathbb{Q}$-Cartier $\mathbb{Q}$-divisor on a normal projective variety $X$. The {\it Kodaira dimension} of $D$ is defined to be
{\small $$
\kappa(X, D)=\begin{cases}
\max\{k\in \mathbb{Z}_{\geq 0}\mid \varlimsup\limits_{m\to \infty}m^{-k}{h^0(X,\lfloor mD\rfloor )}>0\}, & \text{if\ } |\lfloor mD\rfloor|\neq\emptyset\\
&\text{for a\ }m\in \mathbb{Z}_{> 0}; \\
-\infty, & \text{otherwise}.
\end{cases}$$}
\end{defn}

For a normal projective variety $X$ such that $K_X$ is $\mathbb{Q}$-Cartier, denote $\kappa(X)=\kappa(X, K_X)$. Note that if $X$ has at worst canonical singularities, then $\kappa(X)=\dim X$ if and only if $X$ is of general type.

\begin{defn}
Let $D$ be a nef $\mathbb{Q}$-Cartier $\mathbb{Q}$-divisor on a projective variety $X$. The {\it numerical Kodaira dimension} of $D$ is defined to be 
$$
\nu(X, D):=\max\{k\in \mathbb{Z}_{\geq 0}\mid D^k\not\equiv 0\}.
$$
\end{defn}

For a normal projective variety $X$ such that $K_X$ is $\mathbb{Q}$-Cartier and nef, denote $\nu(X)=\kappa(X, K_X)$. 
The famous abundance conjecture states that, if $X$ is a normal projective variety with mild singularities (e.g., canonical singularities) such that $K_X$ is $\mathbb{Q}$-Cartier and nef, then $\kappa(X)=\nu(X)$. In particular, this conjecture was proved if $\dim X=3$ and $X$ has canonical singularities (see \cite{K5, Mi4, K6} and references therein).

\subsection{Cyclic quotient singularities}\label{sec quot sing}\

Let $r$ be a positive integer. Denote by ${\bm \mu}_r$ the cyclic group of $r$-th roots of unity in $\bC$. A {\em cyclic quotient singularity} is of the form $\mathbb{A}^{n}/{\bm \mu}_r$, where the action of ${\bm \mu}_r$ is given by 
\[
{\bm \mu}_r\ni \xi: (x_1, \dots, x_n)\mapsto (\xi^{a_1}x_1, \dots,\xi^{a_n} x_n)
\]
for certain $a_1, \dots, a_n\in \mathbb{Z}/r$. Note that we may always assume that the action of ${\bm \mu}_r$ on $\mathbb{A}^{n}$ is small, that is, it contains no reflection (\cite[Definition 7.4.6, Theorem 7.4.8]{Ishii}), which is equivalent to that $\gcd(r, a_1,...,\hat{a}_i,...,a_n)=1$ for every $1\leq i\leq n$ by \cite[Remark 1]{Fujiki}. In this case, we say that $\mathbb{A}^{n}/{\bm \mu}_r$ is of {\em type $\frac{1}{r}(a_1,\dots, a_n)$}. We say that $Q\in X$ is a {\em cyclic quotient singularity of type $\frac{1}{r}(a_1,\dots, a_n)$}
if $(Q\in X)$ is locally analytically isomorphic to a neighborhood of $(0\in \mathbb{A}^{n}/{\bm \mu}_r)$. Recall that this singularity is isolated if and only if $\gcd(a_i, r)=1$ for every $1\leq i\leq n$ by \cite[Remark 1]{Fujiki}. 

The toric geometry interpretation of cyclic quotient singularities, by virtue of Reid \cite[(4.3)]{Rei87}, is as follows. Let $\overline{M}\simeq \mathbb{Z}^{n}$ be the lattice of monomials on $\mathbb{A}^{n}$, and $\overline{N}$ its dual. Define $N=\overline{N}+\mathbb{Z}\cdot \frac{1}{r}(a_1, \dots, a_n)$ and $M\subset \overline{M}$ the dual sub-lattice. Let $\sigma=\mathbb{R}_{\geq 0}^{n}\subset N_\mathbb{R}$ be the positive quadrant and $\sigma^\vee \subset M_\mathbb{R}$ the dual quadrant. Then in the language of toric geometry, 
\[
\mathbb{A}^{n}=\text{Spec}~\bC[\overline{M}\cap \sigma^\vee]
\]
and its quotient
\[
\mathbb{A}^{n}/{\bm \mu}_r=\text{Spec}~\bC[{M}\cap \sigma^\vee]=T_N(\Delta),
\]
where $\Delta$ is the fan corresponding to $\sigma$.

We refer to \cite{Rei87} for the definitions of terminal singularities and canonical singularities. Here we only mention a criterion on whether a cyclic quotient singularity is terminal or canonical.

\begin{lem}[{\cite[4.11]{Rei87}}]\label{can lem}
A cyclic quotient singularity of type $\frac{1}{r}(a_1,\dots, a_n)$ is terminal (resp. canonical) if and only if 
$$
\sum_{i=1}^n\big\{\frac{ka_i}{r}\big\}> 1 \, (\text{resp. } \geq 1)
$$
 for $k=1,\dots, r-1$. Here $\big\{\frac{ka_i}{r}\big\}=\frac{ka_i}{r}-\rounddown{\frac{ka_i}{r}}$ for each $i$ and $k$. 
\end{lem}

Also it is well-known that a $3$-dimensional cyclic quotient singularity is terminal if and only if it is of type 
$\frac{1}{r}(1,-1, a)$ with $\gcd(a,r)=1$ (\cite[5.2]{Rei87}).

\subsection{The Reid basket}\

A {\it basket} $B$ is a collection of pairs of integers (permitting
weights), say $\{(b_i,r_i)\mid i=1, \dots, s;\ \gcd(b_i, r_i)=1\}$. For simplicity, we will alternatively write a basket as a set of pairs with weights, 
 say for example,
$$B=\{(1,2), (1,2), (2,5)\}=\{2\times (1,2), (2,5)\}.$$

Let $X$ be a $3$-fold with at worst canonical singularities. According to 
Reid \cite{Rei87}, there is a basket of terminal cyclic quotient singularities (called the {\it Reid basket})
$$B_X=\bigg\{(b_i,r_i)\mid i=1,\dots, s;\ 0<b_i\leq \frac{r_i}{2};\ \gcd(b_i,r_i)=1\bigg\}$$
associated to $X$, where a pair $(b_i,r_i)$ corresponds to a terminal cyclic quotient singularity of type $\frac{1}{r_i}(1,-1,b_i)$. The way of determining the Reid basket of $X$ is to take a terminalization (i.e. a crepant $\mathbb{Q}$-factorial terminal model) $X'\to X$ and to locally deform every terminal singularity of $X'$ into a finite set of terminal cyclic quotient singularities.

In this article we only need to compute the baskets for minimal projective $3$-folds with terminal cyclic quotient singularities, in this case the Reid basket coincides with the set of singular points, where a singular point of type $\frac{1}{r}(1,-1,a)$ is simply denoted as $(a,r)$ under no circumstance of confusion. 

\subsection{Weighted projective spaces and weighted hypersurfaces}\label{subsec wps}\

We refer to \cite{WPS, Fle00} for basic knowledge of weighted projective spaces and weighted hypersurfaces.
\begin{defn}[{\cite[5.11, 6.10]{Fle00}}]\label{wellform}
	\begin{enumerate}
		\item A weighted projective space $\mathbb{P}(a_0,...,a_n)$ is {\it well-formed} if $\gcd(a_0,...,\hat{a}_i,...,a_n)=1$ 
			for each $i$.
		\item A hypersurface $X_d$ in $\mathbb{P}(a_0,...,a_n)$ of degree $d$ is {\it well-formed} if
		$P(a_0,...,a_n)$ is well-formed and
		$\gcd(a_0,...,\hat{a}_i,...,\hat{a}_j,...,a_n)  | d$
		for all distinct $i,\ j$.
	\end{enumerate}
\end{defn}

\begin{defn}[{\cite[3.1.5]{WPS}, \cite[6.1]{Fle00}}]\label{qs}
	A hypersurface $X\subset \mathbb{P}(a_0,...,a_n)$ is {\em quasismooth} if the corresponding affine cone of $X$ in $\mathbb{A}^{n+1}$ is smooth outside the origin point.
\end{defn}
For the quasismoothness of a general weighted hypersurface, we have the following criterion.
\begin{thm}[{cf. \cite[Theorem 8.1]{Fle00}}]\label{2.7}
	Let $n$ be a positive integer. The general hypersurface $X_d\subset \mathbb{P}(a_0,...,a_n)$ of degree $d$ is quasismooth if and only if either of the following holds:
	\begin{enumerate}
		\item there exists a variable $x_i$ of weight 
		$d$ for some $i$ (that is, $X$ is a linear cone); 
		\item for every nonempty subset $I=\{i_0,...,i_{k-1}\}$ of 
		$\{0,...,n\}$, 
		either 
			\begin{enumerate}
		\item[(a)] there exists a monomial $x^M_I=x^{m_0}_{i_0}\dots x^{m_{k-1}}_{i_{k-1}}$ of 
		degree $d$, or
		 \item[(b)] for $\mu = 1,...,k$, there exists monomials
		$$x^{M_\mu}_Ix_{e_\mu}=x^{m_{0,\mu}}_{i_0}\dots x^{m_{k-1,\mu}}_{i_{k-1}} x_{e_\mu}$$
		of degree $d$, where $\{e_\mu\}$ are $k$ distinct elements in $\{0,...,n\}$. 
		Here $m_j, m_{j, \mu}$ may be $0$.
		\end{enumerate}
	\end{enumerate}
\end{thm}

For a well-formed and quasismooth $X_d\subset \mathbb{P}(a_0,...,a_n)$, $h^0(X, \mathcal{O}_X(m))$ can be computed by counting monomials of degree $m$ as in \cite[3.4.3]{WPS} or \cite[Lemma~7.1]{Fle00}. Also we will ofter use the fact that the self-intersection number $\mathcal{O}_X(1)^n=\frac{d}{\prod_{i=0}^{n}a_i}$.

\subsection{Singularities on weighted hypersurfaces}\

Singularities of a well-formed quasismooth weighted hypersurface can be determined easily by looking at its defining equations. We refer to \cite[Sections 9-10]{Fle00} for the general method. Here we illustrate the result on singularities of $3$-dimensional hypersurfaces.

\begin{prop}[{cf. \cite[Sections 9-10]{Fle00}}]\label{non iso can}
Let $X_d$ be a general well-formed quasismooth $3$-dimensional hypersurface in $\mathbb{P}(a_0,...,a_4)$ of degree $d$. Suppose that $\gcd(a_i, a_j, a_k)=1$ for any distinct $0\leq i,j,k\leq 4$. Then the singularities of $X_d$ only arise along the edges and vertices of $\mathbb{P}(a_0,...,a_4)$. Denote $P_0, \dots, P_4$ to be the vertices. Then the set of singularities of $X_d$ is determined as follows:
\begin{enumerate}
 \item For a vertex $P_i$,
 \begin{enumerate}[label=(1.\roman*)]
 \item if $a_i | d$, then $P_i\not \in X_d$;
 \item if $a_i\nmid d$, then there exists another index $j$ such that $a_i | d-a_j$, and $P_i\in X_d$ is a cyclic quotient singularity of type $\frac{1}{a_i}(a_k, a_l, a_m)$.
 \end{enumerate}
 \item For an edge $P_iP_{j}$ (that is, $\overline{P_iP_{j}}\setminus \{P_i, P_j\}$, where $\overline{P_iP_{j}}$ is the line passing through $P_i$ and $P_j$), denote $e=\gcd(a_i, a_j)$,
 \begin{enumerate}[label=(2.\roman*)]
 \item if $e | d$, then $P_iP_{j}\cap X$ consists of exactly $\lfloor \frac{ed}{a_ia_j}\rfloor$ points, each point is a cyclic quotient singularity of type $\frac{1}{e}(a_k, a_l, a_m)$;
\item if $e\nmid d$, then $P_iP_{j}\subset X$, and there exists another index $k$ such that $e | d-a_k$, in this case, $P_iP_{j}$ is analytically isomorphic to $\mathbb{C}^*\times \frac{1}{e}(a_l, a_m)$, and each point on $P_iP_{j}$ is a cyclic quotient singularity of type $\frac{1}{e}(0, a_l, a_m)$. 
 \end{enumerate}
\end{enumerate}
Here $\{i,j,k,l,m\}$ is a reordering of $\{0,1,2,3,4\}$.
\end{prop}

\subsection{Weighted blow-ups of cyclic quotient singularities}\

Weighted blow-ups of cyclic quotient singularities play important roles in our construction of new examples. As cyclic quotient singularities are toric, certain blow-ups can be constructed and computed easily using toric geometry. We recall the following proposition from \cite{And18}.
\begin{prop}[{cf. \cite{And18}}]\label{wb}
	Let $X$ be a normal variety of dimension $n$ such that $K_X$ is $\mathbb{Q}$-Cartier.
Suppose that $X$ has a cyclic quotient singularity $Q$ of type $\frac{1}{r}(a_1,a_2,\dots, a_n)$, where $a_1>0$, $a_2>0$,$\cdots$, $a_n>0$, $\gcd(a_1,a_2,\dots, a_n)=1$. Then we can take a weighted blow-up $\pi:Y\to X$, at $Q$ with weight $(a_1,a_2,\dots, a_n)$, which has the following properties:
\begin{enumerate}
 \item The exceptional divisor $\pi^{-1}(Q)=E\cong \mathbb{P}(a_1,a_2,\dots, a_n)$.
 \item $\OO_Y(E)|_E\cong \OO_{\mathbb{P}(a_1,a_2,\dots, a_n)}(-r)$.
 \item Locally over $Q$, $Y$ is covered by $n$ affine pieces of cyclic quotient singularities of types $\frac{1}{a_i}(-a_1,\dots, r,\dots, -a_n)$, which is obtained by replacing the $i$-th term of $(-a_1,\dots, -a_n)$ with $r$ for each $i$. 
 \item $K_Y=\pi^*(K_X)-\frac{r-\sum_{i=1}^n a_i}{r}E$.
\end{enumerate}
In particular, if $X$ is projective and $\mathbb{P}(a_1,a_2,\dots, a_n)$ is well-formed, then
$$
K_Y^n=\Big(\pi^*(K_X)-\frac{r-\sum_{i=1}^n a_i}{r}E\Big)^n=K_X^n-\frac{(r-\sum_{i=1}^n a_i)^n}{r\prod_{i=1}^na_i}.
$$
\end{prop}
\begin{proof} 
 For Properties (1)$\sim$(4), we refer to \cite[Section 2]{And18}. In fact, everything can be treated locally using toric geometry. The last statement is simply a direct consequence. In fact, as $E$ is contracted to a point, any intersections involving both $\pi^*(K_X)$ and $E$ are zero. On the other hand, $E^n=(E|_E)^{n-1}=\frac{(-r)^{n-1}}{\prod_{i=1}^na_i}$ by (2).
\end{proof}

\section{The nefness criterion}\label{sec 3}
 
 In this section, we start by looking into a well-formed quasismooth weighted hypersurface $X$ with isolated singularities, among which only one singular point, say $Q$, is non-canonical. If we take a partial resolution, by means of a weighted blow-up at $Q\in X$, to get the birational morphism $Y \to X$, $Y$ would have milder singularities than $X$ does. A very natural question is whether $K_Y$ is nef. If so, this will give us an easy but new strategy for constructing minimal varieties. A key observation of this section is Theorem~\ref{nefness} (which we call the ``nefness criterion''). As applications of Theorem~\ref{nefness}, in the next section, we 
 construct 59 families of new minimal $3$-folds of general type, in Table~\ref{tableA}, Table~\ref{tableAp}, Table~\ref{tableC}, and Table~\ref{tableC+}, followed by infinite families of minimal $3$-folds of Kodaira dimension $2$ (see Tables~\ref{tab kod 2 6r}$\sim$\ref{tab kod 2 4r+6}). All these minimal $3$-folds can naturally evolve into higher dimensional minimal varieties (see Section~\ref{sec 6}). 

%\subsection{Nefness criterion}

%\begin{thm}[Nefness criterion]\label{nef pbf}
%Let $X_d\subset \mathbb{P}(b_1, \dots, b_{n+2})$ be an $n$-dimensional quasismooth general hypersurface of degree $d$ with $\alpha=d-\sum_{i=1}^{n+2}b_i>0$, here $b_i$ are not necessarily ordered by size. Denote the coordinates of $\mathbb{P}(b_1, \dots, b_{n+2})$ by $x_1,\dots,x_{n+2}$. Suppose that $X$ has a cyclic quotient singularity at the point $Q=(x_1=x_2=\dots =x_{n}=0)$ of type $\frac{1}{r}(e_1,\dots,e_n)$. Here $e_1, \dots, e_n>0$, $\gcd(e_1, \dots, e_n)=1$, $\sum_{i=1}^ne_i<r$, and $x_1,\dots,x_n$ are also the local coordinates of $Q$ corresponding to the weights $\frac{e_1}{r},\dots,\frac{e_n}{r}$ respectively. Take $\pi: Y\to X$ to be the weighted blow-up at $Q$ with weight $(e_1,\dots,e_n)$.
%Suppose that there exists an index $k\in \{1,\dots,n\}$ such that 
%\begin{enumerate}
%\item $\alpha e_{j} \geq b_{j}(r-\sum_{i=1}^ne_i)$ for each $j\in \{1, \dots, n\}\setminus \{k\}$;
%\item $e_j\alpha\geq b_j(r-e_1-e_2-e_3);$
%\item $K_Y^3> 0$, that is, $$\frac{\alpha^3d}{b_1b_2b_3b_4b_5}-\frac{(r-a-b-c)^3}{rabc}>0.$$
%\item $\alpha dre_{k}\geq b_{k}b_{n+1}b_{n+2}(r-\sum_{i=1}^ne_i);$
%\item a general hypersurface of degree $d$ in $\mathbb{P}(b_{k}, b_{n+1}, b_{n+2})$ is irreducible;
%\item $\mathbb{P}(e_1,\dots,e_n)$ is well-formed.
%\end{enumerate}
%Then $K_Y$ is nef and $\nu(Y, K_Y)\geq n-1$. 
%\end{thm}

\begin{proof}[Proof of Theorem~\ref{nefness}]
Recall that $X$ has a cyclic quotient singularity of type $\frac{1}{r}(e_1,\dots, e_n)$ at $Q$ and $\pi: Y\to X$ be the weighted blow-up at $Q$ with weight $(e_1,\dots, e_n)$ as in Proposition~\ref{wb}.

Without loss of generality, after rearranging of indices, we may assume that the assumptions hold for $ k=n$.
For each $j=1,\dots, n-1$, let $H_j\subset X$ be the effective Weil divisor defined by $x_j=0$, denote $L$ to be a Weil divisor corresponding to $\mathcal{O}_X(1)$. Then $$H_j\sim b_j L\ \text{and}\ K_X\sim \alpha L.$$
Denote $H'_j$ to be the strict transform of $H_j$ on $Y$ and $E$ to be the exceptional divisor of $\pi$. Then 
$$\pi^*H_j=H'_j+\frac{e_j}{r}E.$$ 
Denote $t_j=\frac{\alpha e_j-b_j(r-\sum_{i=1}^ne_i)}{b_jr}$. Then $t_j\geq 0$ for each $j=1,\dots, n-1$ by assumption.
As $K_Y=\pi^*K_X-\frac{r-\sum_{i=1}^ne_i}{r}E$, 
we can see that 
\begin{align}\label{K=H+E}
K_Y\sim_\mathbb{Q} \frac{\alpha}{b_j}H'_j+t_j E
\end{align}
for each $j=1,\dots, n-1$.

Assume, to the contrary, that $K_Y$ is not nef. Then there exists a curve $C$ on $Y$ such that $(K_Y\cdot C)<0$.
Note that $K_Y|_E=\frac{r-\sum_{i=1}^ne_i}{r}(-E)|_E$ is ample, hence $C\not \subset E$. Therefore Equation \eqref{K=H+E} implies that $C\subset \cap_{j=1}^{n-1} H'_j$. 

We claim that $\text{Supp}(\cap_{j=1}^{n-1} H'_j)=C$. It suffices to show that $\text{Supp}(\cap_{j=1}^{n-1} H'_j)$ is an irreducible curve.
Note that $\pi (\text{Supp}(\cap_{j=1}^{n-1} H'_j))=\cap_{j=1}^{n-1} H_j$ is a general hypersurface of degree $d$ in $\mathbb{P}(b_n, b_{n+1}, b_{n+2}),$ hence $\pi (\text{Supp}(\cap_{j=1}^{n-1} H'_j))$ is an irreducible curve by assumption. On the other hand, the support of $\cap_{j=1}^{n-1} H'_j\cap E$ is just the point $[0:\dots : 0: 1]$ in $E\simeq \mathbb{P}(e_1,\dots, e_n)$. So $\text{Supp}(\cap_{j=1}^{n-1} H'_j)$ is just the strict transform of $\pi (\text{Supp}(\cap_{j=1}^{n-1} H'_j))$, which is an irreducible curve.

Therefore, we can write $(H'_1\cdot \dots \cdot H'_{n-1})=t C$ for some $t>0$ as $1$-cycles. Then $(K_Y\cdot C)<0$ implies that 
$(K_Y\cdot H'_1\cdot \dots \cdot H'_{n-1} )<0$. On the other hand, 
\begin{align*}{}&(K_Y\cdot H'_1\cdot \dots \cdot H'_{n-1} )\\={}&((\pi^*K_X-\frac{r-\sum_{i=1}^ne_i}{r}E)\cdot (\pi^*H_1- \frac{e_1}{r}E)\cdot \dots \cdot (\pi^*H_{n-1}- \frac{e_{n-1}}{r}E))\\
={}&\alpha (\prod_{j=1}^{n-1}b_j ) L^{n}+(-1)^n\frac{(r-\sum_{i=1}^ne_i)\prod_{j=1}^{n-1}e_j}{r^n}E^n\\
={}&\frac{\alpha d}{b_nb_{n+1}b_{n+2}}-\frac{r-\sum_{i=1}^ne_i}{re_n}\geq 0,
\end{align*}
a contradiction.

Hence we conclude that $K_Y$ is nef. The fact that $\nu(Y)\geq n-1$ follows from $(K_Y^{n-1}\cdot E)>0$ as $K_Y|_E$ is ample.
\end{proof}

\begin{rem}\label{nefrk}
	Here we mention a special but important case of Theorem~\ref{nefness}. If $\alpha=r-\sum_{i=1}^ne_i$, and there exists an index $k\in \{1,\dots,n\}$ such that 
 $b_j=e_j$ for each $j\in \{1, \dots, n\}\setminus \{k\}$, then condition (1) in Theorem~\ref{nefness} automatically holds and, meanwhile, condition (2) is equivalent to $K_Y^n\geq 0$ as $K_Y^n = \frac{\alpha^n d}{\prod_{j=1}^{n+2}b_j}-\frac{(r-\sum_{i=1}^n e_i)^n}{r\prod_{i=1}^ne_i}$ by Proposition~\ref{wb}.
\end{rem}

As the last part of this section, we provide several lemmas which are helpful for applying Theorem~\ref{nefness} and for computing the invariants of resulting minimal models.
	
	For verifying condition (3) of Theorem~\ref{nefness}, we need to check the irreducibility of a general curve in a weighted projective plane, which can be done using the following lemma.
	
\begin{lem}\label{irreducible}
		Let $C$ be a general hypersurface of degree $d$ in the well-formed space $\mathbb{P}(a,b,c)$. Suppose that, in the weighted polynomial ring $\bC[x,y,z]$ with $\text{\rm weight}\, x=a$, $\text{\rm weight}\, y=b$ and $\text{\rm weight}\, z=c$, 
		\begin{enumerate}
		\item there are at least two monomials of degree $d$;
		\item all monomials of degree $d$ have no common divisor;
		\item the set of monomials of degree $d$ cannot be written as $\{g_1^i g_2^{k-i}\mid i=0,1,\dots, k\}$ for some integer $k>1$, where $k$ divides $d$ and $g_1$, $g_2$ are two monomials of degree $\frac{d}{k}$.
		\end{enumerate}
		Then $C$ is irreducible. 
	\end{lem} 
	
	\begin{proof}
	Note that $C$ is a general member of the linear system $|\OO(d)|$.	By the first two conditions, $|\OO(d)|$ has no base component. Suppose, to the contrary, that $C$ is reducible. Then $h^0(\OO(d))\geq 3$ and $|\OO(d)|$ is composed with a pencil. Suppose that $|\OO(d')|$ is the corresponding irreducible pencil for some positive integer $d'$, then $d=d'k$ for some integer $k>1$. Moreover, $H^0(\mathbb{P}(a,b,c), \OO(d'))$ is spanned by two monomials $g_1, g_2$ of degree $d'$, which implies that $H^0(\mathbb{P}(a,b,c), \OO(d))$ is spanned by $\{g_1^i g_2^{k-i}\mid i=0,1,\dots, k\}$, a contradiction.
	\end{proof}
	
	\begin{rem}\label{irreducible remark} It can be checked that condition (3) of Lemma~\ref{irreducible} holds in each of the following cases:
	\begin{enumerate}[label=(3.\roman*)]
	\item There exist 3 monomials of degree $d$ of forms $x^{m_1}y^{n_1}$, $y^{m_2}z^{n_2}$, $x^{m_3}z^{n_3}$, where $m_i, n_i$ are positive integers for $i=1,2,3$.
	\item There exist 2 monomials of degree $d$ of forms $x^{m_1}$, $y^{m_2}z^{m_3}$, such that $\gcd(m_1, m_2, m_3)=1$, where $m_i$ is a non-negative integer for $i=1,2,3$.
	\end{enumerate}
Assume to the contrary that the set of monomials of degree $d$ can be written as $\{g_1^i g_2^{k-i}\mid i=0,1,\dots, k\}$. 
	For Case~(3.i), as we have $3$ monomials, without loss of generality, we may assume that 
	$x^{m_1}y^{n_1}=g_1^i g_2^{k-i}$ for some $0<i<k$. This implies that $z$ does not divide $g_1, g_2$, which contradicts to the existence of $y^{m_2}z^{n_2}$.
 For Case~(3.ii), as $x^{m_1}$ and $y^{m_2}z^{m_3}$ are monomials of degree $d$, then possibly switching $g_1, g_2$ we have $g_1^k=x^{m_1}$ and $g_2^k=y^{m_2}z^{m_3}$, which implies that $k$ divides $m_1, m_2, m_3$, but this is absurd as $\gcd(m_1, m_2, m_3)=1$. 
\end{rem}

	\begin{rem}\label{irreducible remark 2}
	The readers should be warned that when applying Theorem~\ref{nefness}, $\mathbb{P}(b_k, b_{n+1}, b_{n+2})$ may not be well-formed, that is, $b_k, b_{n+1}, b_{n+2}$ may not be coprime to each other. In this case, to apply Lemma~\ref{irreducible}, we should firstly make it well-formed by dividing out common factors. Such a procedure will be illustrated in Example~\ref{ex cj3}.
	\end{rem}

	The next lemma concerns the change of Picard numbers under a crepant blow-up of a canonical cyclic quotient singularity.
	
	\begin{lem}\label{can sing res}
		Let $(X,Q)$ be a germ of $n$-fold isolated cyclic quotient canonical singularity of type $\frac{1}{r}(a_1,\dots,a_n)$. Then there is a terminalization $ X'\rightarrow (X,Q)$ such that 
	 $$\rho(X'/X)=\#\{m\in \mathbb{Z}\mid \sum_{i=1}^n\big\{\frac{ma_i}{r}\big\}=1, 1\leq m \leq r-1\}.$$
	\end{lem}
\begin{proof}
This can be seen by toric geometry. Recall the notation in Subsection~\ref{sec quot sing}. Let $\overline{M}\simeq \mathbb{Z}^{n}$ and $\overline{N}$ its dual. Define $N$ by $N=\overline{N}+\mathbb{Z}\cdot \frac{1}{r}(a_1, \dots, a_n)$ and $M\subset \overline{M}$ the dual sublattice. Let $\sigma=\mathbb{R}_{\geq 0}^{n}\subset N_\mathbb{R}$ be the positive quadrant and $\sigma^\vee \subset M_\mathbb{R}$ the dual quadrant. Then
\[
X=\text{Spec}~\bC[{M}\cap \sigma^\vee]=T_N(\Delta),
\]
where $\Delta$ is the fan corresponding to $\sigma$.
Consider the set of lattice points \begin{align*}
S={}&\{ (x_1,\dots,x_n)\in N\cap \mathbb{R}_{> 0}^{n} \mid \sum_{i=1}^nx_i=1\}\\
={}&\{ (\big\{\frac{ma_1}{r}\big\}, \dots, \big\{\frac{ma_n}{r}\big\})\mid \sum_{i=1}^n\big\{\frac{ma_i}{r}\big\}=1, 1\leq m \leq r-1\}.\end{align*}
Take $\sigma(S)$ to be any subdivision of $\sigma$ by $S$ into simplicial cones, and take $\Delta(S)$ to be the corresponding fan. Then it can be checked by Lemma~\ref{can lem} that $X'=T_N(\Delta(S))$ is the required terminalization, and $\rho(X'/X)=\#S$.
%$(a)$ If $r=2$, the singularity is isomorphic to $\frac{1}{2}(1,1)\times \mathbb{C}$, the assertion clearly follows.
	%
	%Now assume $r\geq 3$, 
	%according to Proposition~\ref{wb}, after taking a weighted blow-up of weight $(1,1,r-2)$: 
	%$$f_1:X_1\rightarrow X$$
	%where $X_1$ is smooth if $r=3$, or has a unique canonical singularity of type $\frac{1}{r-2}(1,1,r-4)$, hence we have $\rho(X_1) = \rho(X)$ 
	%and $\mathbf{B}(X_1) = \mathbf{B}(X)$. Then we repeat this process and get the assertions for $(a)$.
\end{proof}

	\begin{exmp}\label{eg can sing res} We illustrate Lemma~\ref{can sing res} by the following examples in dimension $3$.
	\begin{enumerate}
	\item If $X$ is of type $\frac{1}{r}(1,1,r-2)$ with $r$ odd, then
	$$
	S=\{ \frac{1}{r}(m, m, r-2m)\mid 1\leq m \leq \lfloor \frac{r}{2}\rfloor\}.
	$$
	We can take $\sigma(S)$ to be the subdivision of $\sigma=\cone(e_1, e_2, e_3)$ into $\cone(e_1, e_2, e_{\lfloor \frac{r}{2}\rfloor+3})$, $\cone(e_1, e_{m}, e_{m+1})$, and $\cone(e_2, e_{m}, e_{m+1})$ for $m=3, \dots, \lfloor \frac{r}{2}\rfloor+2$, where $e_1=(1,0,0)$, $e_2=(0,1,0)$, $e_3=(0,0,1)$ and, for $m>3$, $e_m=\frac{1}{r}(m-3, m-3, r-2m+6)$.
	It is easy to check that the resulting $X'$ is smooth and $\rho(X'/X)=\lfloor r/2 \rfloor$.
	
	\item If $X$ is of type $\frac{1}{7}(1,2,4)$, then
	$$
	S=\{ \frac{1}{7}(1,2,4), \frac{1}{7}(2,4,1), \frac{1}{7}(4,1,2)\}.
	$$
	We can take $\sigma(S)$ to be the subdivision of $\sigma=\cone(e_1, e_2, e_3)$ into $\cone(e_1, e_5, e_6)$, $\cone(e_1, e_2, e_5)$, $\cone(e_2, e_4, e_5)$, $\cone(e_2, e_3, e_4)$, 
	$\cone(e_3, e_4, e_6)$, $\cone(e_1, e_3, e_6)$, $\cone(e_4, e_5, e_6)$. Here $e_1=(1,0,0)$, $e_2=(0,1,0)$, $e_3=(0,0,1)$, $e_4= \frac{1}{7}(1,2,4)$, $ e_5=\frac{1}{7}(2,4,1)$, $e_6=\frac{1}{7}(4,1,2)$.
	It is easy to check that the resulting $X'$ is smooth and $\rho(X'/X)=3$.

\item If $X$ is of type $\frac{1}{4}(1,2,3)$, then
	$
	S=\{ \frac{1}{2}(1,0,1)\}.
	$
	We can take $\sigma(S)$ to be the subdivision of $\sigma=\cone(e_1, e_2, e_3)$ into $\cone(e_1, e_2, e_4)$, $\cone(e_2, e_3, e_4)$. Here $e_1=(1,0,0)$, $e_2=(0,1,0)$, $e_3=(0,0,1)$, $e_4= \frac{1}{2}(1,0,1)$.
	It is easy to check that the resulting $X'$ has 2 cyclic quotient singularities of type $\frac{1}{2}(1,1,1)$ and $\rho(X'/X)=1$.
	
	\end{enumerate}
	\end{exmp}

\section{Applications of the nefness criterion in constructing minimal varieties}\label{sec 4}

\subsection{General construction}\label{construction process}\

 In practice, by applying Theorem~\ref{nefness}, it is possible to search numerous minimal varieties by a computer program. The effectivity may follow from the following steps: 
\begin{quote}
Pick up a general weighted hypersurface of dimension $n$, say
$$X=X^n_d\subset \mathbb{P}(a_0,a_1,...,a_{n+1}).$$
\begin{enumerate}%[itemindent=1em] 
 \item[{\bf Step 0.}] Check that $X$ is well-formed and quasismooth by Definition~\ref{wellform} and Theorem~\ref{2.7};
 \item[{\bf Step 1.}] Compute singularities of $X$ and check that $X$ has a unique non-canonical singularity $Q$ by Proposition~\ref{non iso can}.
 \item[{\bf Step 2.}] Verify that $X$ and $Q$ satisfy the conditions of Theorem~\ref{nefness} using Lemma~\ref{irreducible}, thus %by Theorem~\ref{nef pbf}, 
 one obtains a weighted blow-up $f:\widetilde{X}\rightarrow X$ at $Q$ such that $K_{\widetilde{X}}$ is nef;
 \item[{\bf Step 3.}] Compute singularities of $\widetilde{X}$ and check that $\widetilde{X}$ has only canonical singularities by Proposition~\ref{wb};
 \item[{\bf Step 4.}] Take a terminalization $g:\widehat{X}\rightarrow \widetilde{X}$ by Lemma~\ref{can sing res}. 
\end{enumerate}
In the end, if $X$ passes through all above steps, then the resulting 
$\widehat{X}$ is a minimal projective $n$-fold with $\bQ$-factorial terminal singularities.% and, meanwhile, $\widetilde{X}$ is minimal with at worst canonical singularities. 
\end{quote}

%Following Construction~\ref{construction process}, we can obtain many minimal $3$-folds using a computer program.
	
\subsection{Examples of minimal $3$-folds of general type with canonical volume less than $1$}\

%\begin{thm}
%By Construction~\ref{construction process}, there are $46$ families of minimal $3$-folds of general type $\widehat{X}$ with $\Vol(\widehat{X})<1$ obtained from weighted hypersurfaces $X=X_d\subset \mathbb{P}(a_0,a_1,...,a_4)$ with $1\leq \alpha=d-\sum_{i=0}^4\leq 10$ and $10\leq d\leq 100$. Invariants of $X$ and $\widehat{X}$ are listed in Tables~\ref{tableA} and~\ref{tableAp}.
%\end{thm}

%	We start from a well-formed quasismooth hypersurface 
%	$$X_d\subset \mathbb{P}(a_0,a_1,...,a_4),$$
%	which satisfies the nefness criterion (Theorem~\ref{nef pbf}).
%	 Table $A$ and $A^+$ contain minimal $3$-folds with $1\leq \alpha\leq 10$, $10\leq d\leq 100$ and $\Vol<1$, among them many are new examples. 

Aiming at finding minimal $3$-folds with canonical volume less than 1, we take a 
general weighted hypersurfaces, say $X=X_d\subset \mathbb{P}(a_0,a_1,...,a_4)$ with $1\leq \alpha=d-\sum_{i=0}^4a_i\leq 10$ and $10\leq d\leq 100$, and apply Construction~\ref{construction process}. This will output at least 46 families of minimal $3$-folds of general type, which are listed in 
Table~\ref{tableA} and Table~\ref{tableAp}. Table~\ref{tableA} consists of those $X$ with only isolated singularities, while Table~\ref{tableAp} consists of those $X$ with non-isolated singularities. 

 Here we explain the contents of the tables: each row contains a well-formed quasismooth general hypersurface $X=X_d\subset \mathbb{P}(a_0,a_1,...,a_4).$
	 The columns of the tables contain the following information: 
\begin{center}
	\begin{tabular}{r p{10cm}}
		\hline
		$\alpha$:& The amplitude of $X$, i.e., $\alpha=d-\sum a_i$;\\
		$\deg$:& The degree of $X$, which is $d$;\\
		weight:& Weights of $\bP(a_0,a_1,a_2,a_3,a_4)$;\\
		B-weight:& $\frac{1}{r}(e_1, e_2, e_3)$, the unique non-canonical singularity in $X$, to which we apply Theorem~\ref{nefness};\\
		$\Vol$: & the canonical volume of $\widehat{X}$, i.e., $K_{\widehat{X}}^3$;\\
		$P_2$: & $h^0(\widehat{X}, 2K_{\widehat{X}})$;\\
		$\chi$: & The holomorphic Euler characteristic of $\mathcal{O}_{\widehat{X}}$;\\
		$\rho$: & The Picard number of $\widehat{X}$;\\
		basket: & The Reid basket of $\widehat{X}$.\\		
		\hline
	\end{tabular}
\end{center}
\medskip

Here $\Vol(X)=K_{\widehat{X}}^3$ can be computed by Proposition~\ref{wb}. Note that Proposition~\ref{wb} implies that $K_{\widetilde{X}}+\frac{r-e_1-e_2-e_3}{r}E=f^*K_X$. In all listed examples, since $2(e_1+e_2+e_3)>r$, we see that, for $m=1,2$,
$$h^0(\widehat{X}, m K_{\widehat{X}})=h^0({\widetilde{X}}, mK_{\widetilde{X}})=h^0({\widetilde{X}}, \lfloor mf^*(K_{{X}})\rfloor )=h^0({{X}}, mK_{{X}})
$$
where the last item can be computed on $X$ by counting the number of monomials of degree $m\alpha$. Besides, we have 
$$\chi(\mathcal{O}_{\widehat{X}})=\chi(\mathcal{O}_X)=1-h^0(X, K_X)$$ by \cite[Theorem 3.2.4(iii)]{WPS}. By virtue of Proposition~\ref{non iso can}, any singularity of $\widehat{X}$ lies over a vertex $P_i\in X$ or over some point on $\overline{P_iP_j}\cap X$ for some $i$ and $j$. Hence $\rho(\widehat{X})$ and the basket $B_{\widehat{X}}$
%\footnote{note that by, in order to compute the baskets for examples in Table~\ref{tableAp}, we only need to consider the isolated singularities.} 
can be computed using Proposition~\ref{non iso can}, Proposition~\ref{wb}, Lemma~\ref{can sing res}, and Example~\ref{eg can sing res}. 
%Later we will illustrate explicit computations for several examples in the table.
All examples in Table~\ref{tableA} and Table~\ref{tableAp} have been manually verified. In fact, this is not a hard work at all (see examples following the tables). 

The reason that we split into two tables is that all examples in Table~\ref{tableA} have isolated canonical singularities, so we can compute the Picard number $\rho(\widehat{X})$ of the minimal model $\widehat{X}$ easily by Lemma~\ref{can sing res}; on the other hand, all examples in Table~\ref{tableB} have non-isolated canonical singularities, for which $\rho(\widehat{X})$ is harder to compute, so we omit the computations of $\rho(\widehat{X})$ in this table.
		
\medskip
{\tiny
\begin{longtable}{|l|l|l|l|l|l|l|l|l| p{4.5cm} |}
\caption{Minimal $3$-folds of general type, I}\label{tableA}\\
		\hline
			No.&$\alpha$ & deg & weight & B-weight & $\Vol$&$P_2$ & $\chi$&$\rho$ &basket \\ \hline	
\endfirsthead
\multicolumn{5}{l}{{ {\bf \tablename\ \thetable{}} \textrm{-- continued from previous page}}} \\
\hline 
			No.&$\alpha$ & deg & weight & B-weight & $\Vol$&$P_2$ & $\chi$&$\rho$ &basket \\ \hline	
\endhead

 \multicolumn{5}{l}{{\textrm{Continued on next page}}} \\ \hline
\endfoot
\hline 
\endlastfoot

1&1&12&$(1,1,2,2,5)$&$\frac{1}{5}(1,1,2)$&$\frac{1}{2}$&5&-1&2&$ 7\times(1,2) $\\ \hline
2&1&13&$(1,1,2,3,5)$&$\frac{1}{5}(1,1,2)$&$\frac{1}{3}$&4&-1&2&$ 2\times(1,2),\ (1,3) $\\ \hline
3&1&15&$(1,1,2,3,7)$&$\frac{1}{7}(1,2,3)$&$\frac{1}{3}$&4&-1&2&$ 2\times(1,2),\ (1,3) $\\ \hline
4&1&16&$(1,2,2,3,7)$&$\frac{1}{7}(1,2,3)$&$\frac{1}{6}$&3&0&2&$ 2\times(1,3),\ 9\times(1,2) $\\ \hline
5&1&18&$(1,2,3,4,7)$&$\frac{1}{7}(1,2,3)$&$\frac{1}{12}$&2&0&2&$ (1,4),\ (1,3),\ 5\times(1,2) $\\ \hline
6&1&26&$(2,3,4,5,11)$&$\frac{1}{11}(2,3,5)$&$\frac{1}{60}$&1&1&2&$ 2\times(1,3),\ (1,4),\ 2\times(2,5),\ 7\times(1,2) $\\ \hline
7&1&28&$(2,3,4,5,13)$&$\frac{1}{13}(3,4,5)$&$\frac{1}{60}$&1&1&2&$ 2\times(1,3),\ 2\times(2,5),\ (1,4),\ 7\times(1,2) $\\ \hline
8&1&33&$(3,4,5,7,13)$&$\frac{1}{13}(3,4,5)$&$\frac{1}{210}$&0&1&2&$ 2\times(1,4),\ 2\times(2,5),\ (2,7),\ (1,3) $\\ \hline
9&2&35&$(1,3,5,7,17)$&$\frac{1}{17}(3,5,7)$&$\frac{16}{105}$&2&0&2&$ 2\times(1,3),\ (3,7),\ (1,5) $\\ \hline
10&3&28&$(2,2,3,5,13)$&$\frac{1}{13}(2,3,5)$&$\frac{9}{10}$&5&0&4&$ 2\times(1,5),\ 15\times(1,2) $\\ \hline
11&3&30&$(2,3,4,5,13)$&$\frac{1}{13}(2,3,5)$&$\frac{9}{20}$&3&0&3&$ (1,4),\ (1,5),\ 8\times(1,2) $\\ \hline
12&3&33&$(1,4,5,7,13)$&$\frac{1}{13}(1,4,5)$&$\frac{27}{70}$&3&0&2&$ 2\times(1,4),\ 2\times(2,5),\ (3,7) $\\ \hline
13&3&35&$(1,2,5,7,17)$&$\frac{1}{17}(2,5,7)$&$\frac{27}{35}$&5&-1&2&$ 2\times(1,2),\ (2,7),\ (1,5) $\\ \hline
14&3&35&$(1,4,5,7,15)$&$\frac{1}{15}(1,4,7)$&$\frac{27}{70}$&3&0&2&$ 2\times(1,4),\ (3,7),\ 2\times(2,5) $\\ \hline
15&3&36&$(1,4,6,7,15)$&$\frac{1}{15}(1,4,7)$&$\frac{9}{28}$&3&0&3&$ 2\times(3,7),\ (1,4),\ 3\times(1,2) $\\ \hline
16&3&40&$(2,3,5,7,20)$&$\frac{1}{7}(3,1,2)$&$\frac{7}{30}$&2&0&3&$ 3\times(1,2),\ (1,3),\ 2\times(1,5) $\\ \hline
17&3&40&$(2,4,5,7,19)$&$\frac{1}{19}(4,5,7)$&$\frac{27}{140}$&2&1&2&$ 2\times(2,7),\ (2,5),\ (1,4),\ 10\times(1,2) $\\ \hline
18&3&42&$(2,5,7,8,17)$&$\frac{1}{17}(2,5,7)$&$\frac{27}{280}$&1&1&2&$ 2\times(1,5),\ (3,8),\ (2,7),\ 6\times(1,2) $\\ \hline
19&4&40&$(2,3,5,7,19)$&$\frac{1}{19}(3,5,7)$&$\frac{64}{105}$&3&0&2&$ 2\times(1,3),\ 2\times(2,7),\ (2,5) $\\ \hline
20&4&49&$(3,5,7,11,19)$&$\frac{1}{19}(3,5,7)$&$\frac{128}{1155}$&1&1&2&$ 2\times(1,3),\ 2\times(2,5),\ (5,11),\ (2,7) $\\ \hline
21&5&42&$(2,3,7,8,17)$&$\frac{1}{17}(2,3,7)$&$\frac{125}{168}$&4&0&2&$ (3,8),\ (3,7),\ (1,3),\ 6\times(1,2) $\\ \hline
22&5&42&$(2,4,5,7,19)$&$\frac{1}{19}(2,5,7)$&$\frac{25}{28}$&4&0&6&$ (1,4),\ (1,7),\ 11\times(1,2) $\\ \hline
23&5&48&$(2,7,9,12,13)$&$\frac{1}{13}(3,4,5)$&$\frac{383}{1260}$&1&1&2&$ (1,7),\ (2,9),\ (2,5),\ (1,4),\ 2\times(1,3),\ 4\times(1,2) $\\ \hline
24&5&60&$(2,5,7,11,30)$&$\frac{1}{11}(3,5,1)$&$\frac{29}{105}$&2&0&6&$ (1,7),\ (2,5),\ (1,3),\ 2\times(1,2) $\\ \hline
25&7&60&$(4,5,6,11,27)$&$\frac{1}{27}(4,5,11)$&$\frac{343}{660}$&2&1&2&$ 2\times(3,11),\ (2,5),\ (1,4),\ 5\times(1,2),\ (1,3) $\\ \hline
26&7&70&$(2,7,9,10,35)$&$\frac{1}{5}(1,1,2)$&$\frac{4}{9}$&3&0&8&$ (2,9),\ 8\times(1,2) $\\ \hline
27&7&70&$(3,5,7,13,35)$&$\frac{1}{13}(5,3,2)$&$\frac{13}{30}$&2&0&9&$ (1,3),\ (1,2),\ 3\times(1,5) $\\ \hline
28&9&60&$(3,5,11,12,20)$&$\frac{1}{11}(1,4,3)$&$\frac{9}{10}$&3&0&8&$ 2\times(1,4),\ 3\times(2,5) $\\ \hline
29&9&60&$(5,7,11,12,16)$&$\frac{1}{11}(2,5,3)$&$\frac{989}{1680}$&1&1&2&$ (1,7),\ (1,3),\ (2,5),\ (1,2),\ (5,16)$, $(1,4)$\\ \hline
30&9&90&$(2,5,11,18,45)$&$\frac{1}{11}(3,2,5)$&$\frac{11}{15}$&4&0&6&$ 3\times(2,5),\ 6\times(1,2),\ (1,3) $\\ \hline
31&10&60&$(5,7,11,12,15)$&$\frac{1}{7}(1,1,3)$&$\frac{16}{33}$&2&0&10&$ 2\times(1,3),\ (3,11) $\\

\end{longtable}
}

{\tiny
\begin{longtable}{|l|l|l|l|l|l|l|l| p{4.5cm} |}
	\caption{Minimal $3$-folds of general type, II}\label{tableAp}\\
		\hline
		No.&$\alpha$ & deg & weight & B-weight & $\Vol$&$P_2$ & $\chi$ &basket \\ \hline	
\endfirsthead
\multicolumn{5}{l}{{ {\bf \tablename\ \thetable{}} \textrm{-- continued from previous page}}} \\
\hline 
		No.&$\alpha$ & deg & weight & B-weight & $\Vol$&$P_2$ & $\chi$ &basket \\ \hline	
\endhead

 \multicolumn{5}{l}{{\textrm{Continued on next page}}} \\ \hline
\endfoot
\hline 
\endlastfoot

1&2&21&$(1,2,3,5,8)$&$\frac{1}{8}(1,2,3)$&$\frac{8}{15}$&4&-1&$ (1,5),\ (1,3) $\\ \hline
2&2&21&$(1,2,4,5,7)$&$\frac{1}{5}(1,2,1)$&$\frac{1}{2}$&4&-1&$ 3\times(1,2) $\\ \hline

3&2&25&$(1,3,4,5,10)$&$\frac{1}{10}(1,3,4)$&$\frac{4}{15}$&3&0&$ 2\times(1,3),\ 4\times(1,2),\ 2\times(2,5) $\\ \hline
4&2&29&$(1,3,4,5,14)$&$\frac{1}{14}(3,4,5)$&$\frac{4}{15}$&3&0&$ 2\times(1,3),\ 4\times(1,2),\ 2\times(2,5) $\\ \hline
5&2&35&$(3,4,5,7,14)$&$\frac{1}{14}(3,4,5)$&$\frac{4}{105}$&1&1&$ 2\times(1,3),\ 4\times(1,2),\ (2,5),\ 2\times(3,7) $\\ \hline
6&3&26&$(2,3,5,6,7)$&$\frac{1}{7}(3,1,2)$&$\frac{8}{15}$&3&0&$ (1,5),\ 8\times(1,2),\ (1,3) $\\ \hline
7&3&28&$(1,3,4,6,11)$&$\frac{1}{11}(1,3,4)$&$\frac{3}{4}$&5&-1&$ 5\times(1,2),\ (1,4) $\\ \hline
8&3&32&$(2,3,3,5,16)$&$\frac{1}{5}(1,1,2)$&$\frac{1}{2}$&4&-1&$ 3\times(1,2) $\\ \hline
9&4&35&$(2,3,5,7,14)$&$\frac{1}{14}(2,3,5)$&$\frac{64}{105}$&3&0&$ 2\times(1,3),\ (2,5),\ 2\times(2,7) $\\ \hline
10&4&37&$(2,3,5,7,16)$&$\frac{1}{16}(2,3,7)$&$\frac{64}{105}$&3&0&$ 2\times(1,3),\ (2,5),\ 2\times(2,7) $\\ \hline

11&6&51&$(2,7,8,11,17)$&$\frac{1}{11}(1,4,3)$&$\frac{9}{28}$&2&0&$ (2,7),\ 3\times(1,4) $\\ \hline
12&6&55&$(3,4,5,11,26)$&$\frac{1}{26}(4,5,11)$&$\frac{36}{55}$&3&0&$ 4\times(1,2),\ (4,11),\ (1,5) $\\ \hline

13&8&65&$(2,5,9,13,28)$&$\frac{1}{28}(2,5,13)$&$\frac{512}{585}$&3&0&$ (2,9),\ (4,13),\ (1,5) $\\ \hline
14&9&62&$(5,6,7,8,27)$&$\frac{1}{27}(5,6,7)$&$\frac{243}{280}$&3&1&$ 2\times(2,5),\ 8\times(1,2),\ 2\times(2,7),\ (3,8) $\\ \hline
15&9&70&$(5,6,7,16,27)$&$\frac{1}{27}(5,6,7)$&$\frac{243}{560}$&2&1&$ 7\times(1,2),\ (5,16),\ (2,7),\ (2,5) $\\

\end{longtable}
}

%%%%%%%Apr.17%%%%%%%%%%%%%%%%%%%%%%%%%%%%

Taking a couple of typical examples as follows, we illustrate on how to do the manual verification. 

\begin{exmp}[Table~\ref{tableA}, No.~8] Consider the general hypersurface 
$$X=X_{33}\subset \mathbb{P}(3,4,5,7,13),$$
which is clearly well-formed and quasismooth by Definition~\ref{wellform} and Proposition~\ref{2.7}. One also knows that $\alpha = 1$, $p_g = 0$, $P_2=0$, and $\chi(\mathcal{O}_X) = 1$. The set of singularities of $X$ is
	$$\text{\rm Sing}(X) = \{\frac{1}{4}(3,3,1),\ \frac{1}{5}(3,4,1),\ \frac{1}{7}(6,1,5),\ Q=\frac{1}{13}(3,4,5)\},$$
	where the first $3$ singularities are terminal, while the last one is non-canonical.
	
	Here we illustrate how to use Proposition~\ref{non iso can} to determine $\text{\rm Sing}(X)$. 
	Denote $P_0, \dots, P_4$ to be the vertices. As $3 | 33$,  $P_0\not \in X$ while $P_{1}, \dots, P_4\in X$. For $P_4$, we have $13 | 33-7$, so $P_4$ is a singularity of type $\frac{1}{13}(3,4,5)$. Similarly we can determine the types of $P_1, P_2, P_3$. As $3, 4, 5, 7, 13$ are pairwisely coprime, there are no singular points on $P_iP_j\cap X$.

	For applying Theorem~\ref{nefness}, we take 
	$$(b_1,b_2,b_3,b_4,b_5) = (3,4,5,7,13),$$
	$$(e_1,e_2,e_3)=(3,4,5),$$
	$r=13$, and $k=1$. Conditions $(1),(2),(4)$ follow from direct computations (or Remark~\ref{nefrk}). Condition $(3)$ means to check that a general curve $C_{33}\subset\mathbb{P}(3,7,13)$ is irreducible, which follows 
	immediately from Lemma~\ref{irreducible} and Remark~\ref{irreducible remark}.
	
	So by Theorem~\ref{nefness}, we can take a weighted blow-up $f:\widetilde{X}\to X$ at the point $Q$ with weight $(3,4,5)$ such that $K_{\widetilde{X}}$ is nef. On the exceptional divisor $E$ there are $3$ new singularities:
	$$\frac{1}{3}(2,1,2),\ \frac{1}{4}(3,3,1),\ \frac{1}{5}(3,4,2),$$ all of which are terminal. Hence $\widetilde{X}$ is a minimal $3$-fold and $\widehat{X}=\widetilde{X}$. Applying the volume formula for weighted blow-ups (cf. Proposition~\ref{wb}), we get $\Vol(\widehat{X})=K_{\widetilde{X}}^3 = \frac{1}{210}$.
	Since $\rho(X)=1$ by \cite[Theorem 3.2.4(i)]{WPS}, after one weighted blow-up the Picard number becomes $2$.
	 Finally, we collect the singularities of $\widehat{X}$ and obtain the Reid basket 
	$$B_{\widehat{X}} = \{ (1,3),\ 2\times(1,4),\ 2\times(2,5),\ (2,7)\}.$$\qed
\end{exmp}

%Let us go on doing another exercise. 

\begin{exmp}[Table~\ref{tableA}, No.~27] Consider the general hypersurface
	$$X=X_{70}\subset \mathbb{P}(3,5,7,13,35), $$
which is well-formed and quasismooth by Definition~\ref{wellform} and Proposition~\ref{2.7}.
	It is clear that $\alpha = 7$, $P_2 = 2$ and $\chi(\mathcal{O}_X) = 0$. Moreover
	$$\text{\rm Sing}(X) = \{\frac{1}{3}(2,1,2),\ 2\times \frac{1}{5}(1,4,1),\ 2\times\frac{1}{7}(1,4,2),\ Q = \frac{1}{13}(5,3,2)\}.$$

	Here we illustrate how to use Proposition~\ref{non iso can} to determine $\text{\rm Sing}(X)$. 
	Denote $P_0, \dots, P_4$ to be the vertices. Note that  $P_0, P_3  \in X$ while $P_{1}, P_2, P_4\not\in X$. For $P_3$, we have $13 | 70-5$, so $P_4$ is a singularity of type $\frac{1}{13}(3,7,35)=\frac{1}{13}(3, 7, 9)=\frac{1}{13}(5, 3, 2)$. Similarly we can determine the type of $P_0$. For edges, there are singular points on $P_1P_4\cap X$ and $P_2P_4\cap X$. For $P_1P_4\cap X$, there are $\lfloor \frac{5\times 70}{5\times 35}\rfloor=2$ singular points of type $\frac{1}{5}(3,7,13)=\frac{1}{5}(1,4,1)$;  for $P_2P_4\cap X$, there are $\lfloor \frac{7\times 70}{7\times 35}\rfloor=2$ singular points of type $\frac{1}{7}(3,5,13)=\frac{1}{7}(1,4,2)$.

	Take $ (b_1,...,b_5) = (3,7,35,5,13)$, $(e_1,e_2,e_3) = (5,3,2)$, 
	$r = 13$, and $k = 3$. One can check that the conditions of Theorem~\ref{nefness} are satisfied and hence, after a weighted blow-up with weight $(5,3,2)$ at $Q$, we get
	$f:\widetilde{X}\rightarrow X$
	so that $K_{\widetilde{X}}$ is nef and $K^3_{\widetilde{X}}=\frac{13}{30}$.	By Proposition~\ref{wb}, we see that
	$$\text{\rm Sing}(\widetilde{X}) = \{\frac{1}{2}(1,1,1),\ \frac{1}{3}(2,1,2),\ 3\times \frac{1}{5}(1,4,1),\ 2\times\frac{1}{7}(1,4,2),\ 
	\frac{1}{3}(1,1,1)\}.$$
 These are all canonical singularities, among which three are non-terminal: 
 $$\{2\times\frac{1}{7}(1,2,4),\ 
 \frac{1}{3}(1,1,1)\}.$$
	By Lemma~\ref{can sing res} and Example~\ref{eg can sing res}, there exists a terminalization $g:\widehat{X}\rightarrow\widetilde{X}$
	such that $$B_{\widehat{X}}=B_{\widetilde{X}}
	= \{(1,2),\ (1,3),\ 3\times(1,5)\} $$	
	and $\rho(\widehat{X}) = \rho(\widetilde{X})+ \lfloor 3/2 \rfloor +2\times 3 = 9$.
Since $g$ is crepant, $K_{\widehat{X}}^3=K^3_{\widetilde{X}}=\frac{13}{30}$. 
\qed
\end{exmp}

 \begin{rem} Note that all examples found by Iano-Fletcher \cite{Fle00} are of Picard number $1$ by \cite[Theorem 3.2.4(i)]{WPS}, whereas ours are of Picard number at least $2$, so our examples are birationally different from those of Iano-Fletcher, as birationally equivalent minimal varieties have the same Picard number (\cite[Theorem 3.52(2)]{K-M}). On the other hand, most of our examples have different deformation invariants (e.g., $\Vol(\widehat{X})$, $P_2$, or basket) from those of Iano-Fletcher's, except for Table~\ref{tableA}, No.~1-8 and Table~\ref{tableAp}, No.~2 \& 10, of which the invariants coincide with those of certain example in \cite[Table~3]{Fle00}. Probably the two $3$-folds in each of these counterparts are mutually deformation equivalent. 
\end{rem}
	 
%%%%%%%%%%%%%%%checked may 18, 5pm

%\begin{exmp} Consider the general hyperplane $Z_{18}\subset \bP(1,2,3,4,7)$, quasismooth, only one non-canonical singularity of type $\frac{1}{7}(1,2,3)$.
%$\frac{1}{12}=\frac{3}{28}-\frac{7-6}{42}\leq \Vol(Z)
%\leq \frac{3}{28}$. 
%\end{exmp}

%\begin{exmp} $X_{36}\subset \bP(1,4,5,7,16)$ quasismooth.
%Singularities: $1/5(4,2,1)$; $2\times \frac{1}{4}(1,-1,1)$; $\frac{1}{7}(5,1,6)$;
%$\frac{1}{16}(1,5,7)$ (non-canonical). After weighted blow-up, the model $Y$ has 
%datum: $(1,2)$, $2\times (1,3)$, $2\times (1,4)$, $(2,5)$, $(2,7)$, $P_2=3$, $\chi(\mathcal{O}_X)=0$. The expected volume is $\frac{4}{105}$. 
%\end{exmp}

\subsection{Examples of minimal $3$-folds of Kodaira dimension $2$}\

As a direct consequence of Theorem~\ref{nefness}, we can construct infinite series of minimal $3$-folds of Kodaira dimension $2$. 

\begin{lem}\label{construction kod 2}
Let $X_d\subset \mathbb{P}(b_1, b_2, b_3, b_4, b_5)$ be a $3$-dimensional well-formed quasismooth general hypersurface of degree $d$ with $\alpha=d-\sum_{i=1}^{5}b_i>0$ where $b_1,\dots, b_5$ are not necessarily ordered by size. Denote by $x_1$, $x_2$, $x_3$, $x_4$, $x_{5}$ the homogenous coordinates of $\mathbb{P}(b_1, b_2, b_3, b_4, b_5)$. Suppose that $X$ has a cyclic quotient singularity at the point $Q=(x_1=x_2=x_{3}=0)$ of type $\frac{1}{r}(e_1,e_2, e_3)$ where positive integers $e_1$, $e_2$, $e_3$ are coprime to each other, $\sum_{i=1}^3e_i<r$ and $x_1$, $x_2$, $x_3$ are the local coordinates of $Q$ corresponding to the weights $\frac{e_1}{r}$, $\frac{e_2}{r}$, $\frac{e_3}{r}$ respectively. Let $\pi: Y\to X$ be the weighted blow-up at $Q$ with weight $(e_1,e_2,e_3)$.
Assume that $\{d, (b_i)_{i=1}^5, (e_i)_{i=1}^3\}$ belongs to one of the cases listed in Table~\ref{tab kod 2 possible}:
{\normalsize
\begin{longtable}{c|c|c}
\caption{}\label{tab kod 2 possible}\\
\hline \endfirsthead\hline \endhead
\hline\endfoot\hline \endlastfoot
$d$ & $(b_1, b_2, b_3, b_4, b_5)$ & $(e_1,e_2, e_3)$\\\hline
$6r$ & $(a,b,c,2r,3r)$ & $(a,b,c)$\\\hline
$3r+3k$ & $(a,b,r+k,3k,r) $ & $(a,b,k)$\\\hline
${4r+2k}$& $(a,b,2r+k,2k,r)$ & $(a,b,k)$\\
\end{longtable}}
\noindent
Then $K_Y$ is nef and $\nu(Y)=2$. 
\end{lem} 
\begin{proof}
We check that conditions in Theorem~\ref{nefness} hold for $k=3$.
Condition~(1) in Theorem~\ref{nefness} holds since $\alpha=d-\sum_{i=1}^{5}b_i=r-e_1-e_2-e_3$, $b_1=e_1$ and $b_2=e_2$.
Condition~(2) in Theorem~\ref{nefness} holds since $dre_3=b_3b_4b_5$.
One may check that condition~(3) in Theorem~\ref{nefness} holds by Lemma~\ref{irreducible} and Remark~\ref{irreducible remark}.
Condition~(4) in Theorem~\ref{nefness} holds since $e_1$, $e_2$, $e_3$ are coprime to each other.
Hence $K_Y$ is nef and $\nu(Y)\geq 2$ by Theorem~\ref{nefness}.
Note that by Proposition~\ref{wb},
$$
K_Y^3=\frac{d\alpha^3}{b_1b_2b_3b_4b_5}-\frac{(r-e_1-e_2-e_3)^3}{re_1e_2e_3}=0.
$$
Hence $\nu(Y)= 2$.
\end{proof}

In particular, we have the following examples.

\begin{thm}\label{new examples kod 2 summary}
There are examples of infinite series of families of minimal $3$-folds of Kodaira dimension $2$ by Construction~\ref{construction process} and Lemma~\ref{construction kod 2} as the following:
\begin{enumerate}
 \item examples obtained from $X_{6r}\subset \bP(a,b,c,2r,3r)$ with $a,b,c \leq 6$ are listed in Table~\ref{tab kod 2 6r};
 \item examples obtained from $X_{3r+3}\subset \bP(a,b,r+1,3,r)$ with $a,b\leq 7$ are listed in Table~\ref{tab kod 2 3r+3};
 \item examples obtained from $X_{3r+6}\subset \bP(a,b,r+2,6,r)$ with $a,b \leq 5$ are listed in Table~\ref{tab kod 2 3r+6};
 \item examples obtained from $X_{4r+2}\subset \bP(a,b,2r+1,2,r)$ with $a,b \leq 9$ are listed in Table~\ref{tab kod 2 4r+2};
 \item examples obtained from $X_{4r+4}\subset \bP(a,b,2r+2,4,r)$ with $a,b \leq 5$ are listed in Table~\ref{tab kod 2 4r+4};
 \item examples obtained from $X_{4r+6}\subset \bP(a,b,2r+3,6,r)$ with $a,b \leq 5$ are listed in Table~\ref{tab kod 2 4r+6}.
\end{enumerate}
Here we consider general hypersurfaces with only isolated singularities.
%Tables~\ref{tab kod 2 6r}$\sim$\ref{tab kod 2 4r+6} are exhibited following the proof of this theorem. 
Several explicit examples with invariants computed are listed in Table~\ref{tableB}, Appendix~\ref{appendix}.
\end{thm}
 
 \begin{proof}
 For a hypersurface $X_d\subset \mathbb{P}(b_1, b_2, b_3, b_4, b_5)$ listed in Lemma~\ref{construction kod 2}, to apply Construction~\ref{construction process}, it suffices to check that 
 \begin{enumerate}
 \item $X$ is well-formed and quasismooth, 
 \item $X$ has only one non-canonical singularity $Q$, and 
 \item after a weighted blow-up at $Q$, $\widetilde{X}$ has only canonical singularities.
 \end{enumerate}
 These can be verified manually as long as we put certain conditions on $r$. Here we only illustrate the procedure by 2 examples, while other series can be verified in a similar way. Here $\modr(r,k)$ means the smallest non-negative residue of $r$ modulo $k$.
 
Consider the general hypersurface $X=X_{6r}\subset \mathbb{P}(3,4,5,2r,3r)$ (Table~\ref{tab kod 2 6r}, No.~14). Note that $X$ is well-formed if and only if $\gcd(r, 2)=\gcd(r, 3)=\gcd(r,5)=1$.
Also $X$ is quasismooth if and only if $5$ divides one of $\{6r-3,6r-4,r\}$. Therefore, $X$ is both well-formed and quasismooth if and only if $\gcd(r, 2)=\gcd(r, 3)=1$, $\modr(r, 5)\in \{3,4\}$.
On the other hand, $\alpha=r-12$. So a necessary condition for $X$ satisfying Construction~\ref{construction process} is that $r>12$, $\gcd(r, 2)=\gcd(r, 3)=1$, $\modr(r, 5)\in \{3,4\}$. Then we can show that this is actually a sufficient condition by computing singularities. In fact, when $\modr(r, 5)=3$, 
$$
\text{\rm Sing}(X)=\{2\times \frac{1}{3}(1,2,2r), \frac{1}{4}(3,1,3r), \frac{1}{5}(4,1,4), \frac{1}{2}(1,1,1), Q=\frac{1}{r}(3,4,5)\}
$$
which has only one non-canonical singularity $Q$, and after a weighted blow-up at $Q$ with weight $(3,4,5)$, $Q$ is replaced by $3$ terminal cyclic singularities of types
$$
 \{\frac{1}{3}(r, 2, 1), \frac{1}{4}(1, r, 3), \frac{1}{5}(2,1, 3)\}.
$$
Similarly, when $\modr(r, 5)=4$, 
$$
\text{\rm Sing}(X)=\{2\times \frac{1}{3}(1,2,2r),\frac{1}{4}(3,1,3r), \frac{1}{5}(3,3,2), \frac{1}{2}(1,1,1), Q=\frac{1}{r}(3,4,5)\}
$$
which has only one non-canonical singularity $Q$, and after a weighted blow-up at $Q$ with weight $(3,4,5)$, $Q$ is replaced by $3$ terminal cyclic singularities of types
$$
 \{\frac{1}{3}(r, 2, 1), \frac{1}{4}(1, r, 3), \frac{1}{5}(2,1, 4)\}.
$$

Consider another general hypersurface $X=X_{3r+3}\subset \mathbb{P}(5,7,r+1,3,r)$ (Table~\ref{tab kod 2 3r+3}, No.~15). 
Note that $X$ is always well-formed. Also $X$ is quasismooth if and only if $5$ divides one of $\{r+1, 3r-4, r, 2r+3\}$ and $7$ divides one of $\{3r-2, r+1, r, 2r+3\}$. Therefore, $X$ is both well-formed and quasismooth if and only if $\modr(r, 5)\in \{0, 1, 3,4\}$ and $\modr(r, 7)\in \{0, 2, 3, 6\}$. As we require that $X$ has only isolated singularities, we have $\modr(r, 5)\neq 0$ and $\modr(r, 7)\neq 0$.
Note that $X$ has a non-canonical singularity $Q=\frac{1}{r}(5,7,1)$ and after a weighted blow-up at $Q$ with weight $(5,7,1)$, $Q$ is replaced by $2$ cyclic singularities of types
$$
 \{\frac{1}{5}(r, 3, 4), \frac{1}{7}(2, r, 6)\}.
$$
These two singularities are not both canonical if $\modr(r, 5)=4$ or $\modr(r, 7)\in \{2, 3\}$. On the other hand, $\alpha=r-13$. So a necessary condition for $X$ satisfying Construction~\ref{construction process} is that $r>13$, $\modr(r,5)\in\{1,3\}$, $\modr(r,7)=6$. Then we can show that this is actually a sufficient condition by computing singularities. In fact, when $\modr(r, 5)=1$, 
$$
\text{\rm Sing}(X)=\{\frac{1}{5}(2,2,3), \frac{1}{7}(5,3,6), Q=\frac{1}{r}(5,7,1)\}
$$
which has only one non-canonical singularity $Q$, and after a weighted blow-up at $Q$ with weight $(5,7,1)$, $Q$ is replaced by $2$ canonical cyclic singularities of types
$$
 \{\frac{1}{5}(1, 3, 4), \frac{1}{7}(2, 6, 6)\}.
$$
Similarly, when $\modr(r, 5)=3$, 
$$
\text{\rm Sing}(X)=\{\frac{1}{5}(4,3,3), \frac{1}{7}(5,3,6), Q=\frac{1}{r}(5,7,1)\}
$$
which has only one non-canonical singularity $Q$, and after a weighted blow-up at $Q$ with weight $(3,4,5)$, $Q$ is replaced by $2$ canonical cyclic singularities of types
$$
 \{\frac{1}{5}(3, 3, 4), \frac{1}{7}(2, 6, 6)\}.
$$
 \end{proof}
 
The description of the contents of the following tables is as follows. Each row contains a well-formed quasismooth hypersurface 
	 $X=X_d\subset \mathbb{P}(b_1,b_2,...,a_5).$
	 The columns of each table contain the following information: 
\begin{center}
	\begin{tabular}{r p{10cm}}
		\hline
		$\alpha$:& The amplitude of $X$, i.e., $d-\sum a_i$;\\
		$\deg$:& The degree of $X$, which is $d$;\\
		weight:& $(b_1,b_2,b_3,b_4,b_5)$;\\
		B-weight:& $\frac{1}{r}(e_1, e_2, e_3)$, the unique non-canonical singularity to be blown up by applying Theorem~\ref{nefness};\\
		conditions: & Restrictions on $r$, where $\modr(r,k)$ means the smallest non-negative residue of $r$ modulo $k$.\\		
		\hline
	\end{tabular}
\end{center}

% Type $X_{6r}\subset \bP(a,b,c,2r,3r)$ with $a,b,c \leq 6$

{\tiny
\begin{longtable}{|l|l|l|l|l|p{5cm}|}
\caption{Type $X_{6r}\subset \bP(a,b,c,2r,3r)$}\label{tab kod 2 6r}\\
		\hline
		No.&$\alpha$ & deg & weight & B-weight & conditions \\ \hline
\endfirsthead
\multicolumn{5}{l}{{ {\bf \tablename\ \thetable{}} \textrm{-- continued from previous page}}} \\
\hline 
No.&$\alpha$ & deg & weight & B-weight & conditions \\ \hline
\endhead

 \multicolumn{4}{l}{{\textrm{Continued on next page}}} \\ \hline
\endfoot
\hline 
\endlastfoot

1& $r-3$ & $6r$ & $(1,1,1,2r,3r)$ & $\frac{1}{r}(1,1,1)$& $r>3 $\\ \hline %DONE
2& $r-4$ & $6r$ & $(1,1,2,2r,3r)$ & $\frac{1}{r}(1,1,2)$& $r> 4$, $\modr(r,2)\neq 0$\\ \hline
3& $r-5$ & $6r$ & $(1,1,3,2r,3r)$ & $\frac{1}{r}(1,1,3)$& $r> 5$, $\modr(r,3)\neq 0$\\ \hline
4& $r-6$ & $6r$ & $(1,1,4,2r,3r)$ & $\frac{1}{r}(1,1,4)$& $r> 6$, $\modr(r,4)=1$\\ \hline
5& $r-7$ & $6r$ & $(1,1,5,2r,3r)$ & $\frac{1}{r}(1,1,5)$& $r> 7$, $\modr(r,5)=1$\\ \hline
6& $r-8$ & $6r$ & $(1,1,6,2r,3r)$ & $\frac{1}{r}(1,1,6)$& $r> 8$, $\modr(r,2)\neq 0$, $\modr(r,3)\neq 0$\\ \hline
7& $r-6$ & $6r$ & $(1,2,3,2r,3r)$ & $\frac{1}{r}(1,2,3)$& $r> 6$, $\modr(r,2)\neq 0$, $\modr(r,3)\neq 0$\\ \hline
8& $r-8$ & $6r$ & $(1,2,5,2r,3r)$ & $\frac{1}{r}(1,2,5)$& $r> 8$, $\modr(r,5)\in \{1, 2\}$\\ \hline
9& $r-8$ & $6r$ & $(1,3,4,2r,3r)$ & $\frac{1}{r}(1,3,4)$& $r> 8$, $\modr(r,3)\neq 0$, $\modr(r,4)\neq 0$\\ \hline
10& $r-9$ & $6r$ & $(1,3,5,2r,3r)$ & $\frac{1}{r}(1,3,5)$& $r> 9$, $\modr(r,3)\neq 0$, $\modr(r,5)\in \{1, 3\}$\\ \hline
11& $r-10$ & $6r$ & $(1,4,5,2r,3r)$ & $\frac{1}{r}(1,4,5)$& $r> 10$, $\modr(r,4)=1$, $\modr(r,5)\in \{1, 4\}$\\ \hline
12& $r-12$ & $6r$ & $(1,5,6,2r,3r)$ & $\frac{1}{r}(1,5,6)$& $r> 12$, $\modr(r,2)\neq 0$, $\modr(r,3)\neq 0$, $\modr(r,5)=1$\\ \hline
13& $r-10$ & $6r$ & $(2,3,5,2r,3r)$ & $\frac{1}{r}(2,3,5)$& $r> 10$, $\modr(r,2)\neq 0$, $\modr(r,3)\neq 0$, $\modr(r,5)\in\{2,3\}$\\ \hline
14& $r-12$ & $6r$ & $(3,4,5,2r,3r)$ & $\frac{1}{r}(3,4,5)$& $r> 12$, $\modr(r,2)\neq 0$, $\modr(r,3)\neq 0$, $\modr(r,5)\in\{3, 4\}$\\ 
\end{longtable}
}

\medskip

% Type $X_{3r+3}\subset \bP(a,b,r+1,3,r)$ with $a,b\leq 7$
{\tiny
\begin{longtable}{|l|l|l|l|l|p{4.5cm}|}
\caption{Type $X_{3r+3}\subset \bP(a,b,r+1,3,r)$}\label{tab kod 2 3r+3}\\
		\hline
		No.&$\alpha$ & deg & weight & B-weight & conditions \\ \hline
\endfirsthead
\multicolumn{5}{l}{{ {\bf \tablename\ \thetable{}} \textrm{-- continued from previous page}}} \\
\hline 
No.&$\alpha$ & deg & weight & B-weight & conditions \\ \hline
\endhead

 \multicolumn{4}{l}{{\textrm{Continued on next page}}} \\ \hline
\endfoot
\hline 
\endlastfoot
		
		1& $r-3$ & $3r+3$ & $(1,1,r+1,3,r)$ & $\frac{1}{r}(1,1,1)$& $r>3 $\\ \hline 	
		2& $r-4$ & $3r+3$ & $(1,2,r+1,3,r)$ & $\frac{1}{r}(1,2,1)$& $r>4 $, $\modr(r,2)\neq 0$\\ \hline 	
		3& $r-5$ & $3r+3$ & $(1,3,r+1,3,r)$ & $\frac{1}{r}(1,3,1)$& $r>5 $, $\modr(r,3)=1$\\ \hline 	
		4& $r-6$ & $3r+3$ & $(1,4,r+1,3,r)$ & $\frac{1}{r}(1,4,1)$& $r>6 $, $\modr(r,4)=1$\\ \hline 	
		5& $r-7$ & $3r+3$ & $(1,5,r+1,3,r)$ & $\frac{1}{r}(1,5,1)$& $r>7 $, $\modr(r,5)=1$\\ \hline 	
		6& $r-8$ & $3r+3$ & $(1,6,r+1,3,r)$ & $\frac{1}{r}(1,6,1)$& $r>8 $, $\modr(r,6)=1$\\ \hline 	
		7& $r-9$ & $3r+3$ & $(1,7,r+1,3,r)$ & $\frac{1}{r}(1,7,1)$& $r>9 $, $\modr(r,7)=2$\\ \hline 	
		8& $r-6$ & $3r+3$ & $(2,3,r+1,3,r)$ & $\frac{1}{r}(2,3,1)$& $r>6 $, $\modr(r,2)\neq 0$, $\modr(r,3)=1$\\ \hline 	
		9& $r-8$ & $3r+3$ & $(2,5,r+1,3,r)$ & $\frac{1}{r}(2,5,1)$& $r>8 $, $\modr(r,2)\neq 0$, $\modr(r,5)\in \{1,3\}$\\ \hline 	
		10& $r-8$ & $3r+3$ & $(3,4,r+1,3,r)$ & $\frac{1}{r}(3,4,1)$& $r>8 $, $\modr(r,3)=1$, $\modr(r,4)=1$\\ \hline 	
		11& $r-9$ & $3r+3$ & $(3,5,r+1,3,r)$ & $\frac{1}{r}(3,5,1)$& $r>9 $, $\modr(r,3)=1$, $\modr(r,5)\in \{1,4\}$\\ \hline 	
		12& $r-10$ & $3r+3$ & $(4,5,r+1,3,r)$ & $\frac{1}{r}(4,5,1)$& $r>10 $, $\modr(r,4)=1$, $\modr(r,5)\in \{1,2\}$\\ \hline 	
		13& $r-12$ & $3r+3$ & $(4,7,r+1,3,r)$ & $\frac{1}{r}(4,7,1)$& $r>12 $, $\modr(r,4)=1$, $\modr(r,7)=5$\\ \hline 	
		14& $r-12$ & $3r+3$ & $(5,6,r+1,3,r)$ & $\frac{1}{r}(5,6,1)$& $r>12 $, $\modr(r,5)=1$, $\modr(r,6)=1$\\ \hline 	
		15& $r-13$ & $3r+3$ & $(5,7,r+1,3,r)$ & $\frac{1}{r}(5,7,1)$& $r>13 $, $\modr(r,5)\in\{1,3\}$, $\modr(r,7)=6$\\ 	
\end{longtable}
}
\medskip

% Type $X_{3r+6}\subset \bP(a,b,r+2,6,r)$ with $a,b \leq 5$
{\tiny
\begin{longtable}{|l|l|l|l|l|p{4.5cm}|}
\caption{Type $X_{3r+6}\subset \bP(a,b,r+2,6,r)$}\label{tab kod 2 3r+6}\\
		\hline
		No.&$\alpha$ & deg & weight & B-weight & conditions \\ \hline
\endfirsthead
\multicolumn{5}{l}{{ {\bf \tablename\ \thetable{}} \textrm{-- continued from previous page}}} \\
\hline 
No.&$\alpha$ & deg & weight & B-weight & conditions \\ \hline
\endhead

 \multicolumn{4}{l}{{\textrm{Continued on next page}}} \\ \hline
\endfoot
\hline 
\endlastfoot

		1& $r-4$ & $3r+6$ & $(1,1,r+2,6,r)$ & $\frac{1}{r}(1,1,2)$& $r>4 $, $\modr(r,6)=3$\\ \hline 	
		2& $r-6$ & $3r+6$ & $(1,3,r+2,6,r)$ & $\frac{1}{r}(1,3,2)$& $r>6 $, $\modr(r,6)= 5$\\ \hline 	
		3& $r-8$ & $3r+6$ & $(1,5,r+2,6,r)$ & $\frac{1}{r}(1,5,2)$& $r>8 $, $\modr(r,5)\in \{2,3\}$, $\modr(r,6)\in \{1,3\}$\\ \hline 	
		4& $r-10$ & $3r+6$ & $(3,5,r+2,6,r)$ & $\frac{1}{r}(3,5,2)$& $r>10 $, $\modr(r,5)\in \{2,4\}$, $\modr(r,6)=5$\\ 	
\end{longtable}
}
\medskip

 %Type $X_{4r+2}\subset \bP(a,b,2r+1,2,r)$ with $a,b \leq 9$
{\tiny
\begin{longtable}{|l|l|l|l|l|p{4.5cm}|}
\caption{Type $X_{4r+2}\subset \bP(a,b,2r+1,2,r)$}\label{tab kod 2 4r+2}\\
		\hline
		No.&$\alpha$ & deg & weight & B-weight & conditions \\ \hline
\endfirsthead
\multicolumn{5}{l}{{ {\bf \tablename\ \thetable{}} \textrm{-- continued from previous page}}} \\
\hline 
No.&$\alpha$ & deg & weight & B-weight & conditions \\ \hline
\endhead

 \multicolumn{4}{l}{{\textrm{Continued on next page}}} \\ \hline
\endfoot
\hline 
\endlastfoot

		1& $r-3$ & $4r+2$ & $(1,1,2r+1,2,r)$ & $\frac{1}{r}(1,1,1)$& $r>3 $\\ \hline 
		2& $r-4$ & $4r+2$ & $(1,2,2r+1,2,r)$ & $\frac{1}{r}(1,2,1)$& $r>4 $, $\modr(r,2)\neq 0$\\ \hline 	
		3& $r-5$ & $4r+2$ & $(1,3,2r+1,2,r)$ & $\frac{1}{r}(1,3,1)$& $r>5 $, $\modr(r,3)\neq 0$\\ \hline 	
		4& $r-6$ & $4r+2$ & $(1,4,2r+1,2,r)$ & $\frac{1}{r}(1,4,1)$& $r>6 $, $\modr(r,2)\neq 0$\\ \hline 	
		5& $r-7$ & $4r+2$ & $(1,5,2r+1,2,r)$ & $\frac{1}{r}(1,5,1)$& $r>7 $, $\modr(r,5)\in \{1,2\}$\\ \hline 	
		6& $r-8$ & $4r+2$ & $(1,6,2r+1,2,r)$ & $\frac{1}{r}(1,6,1)$& $r>8 $, $\modr(r,6)=1$\\ \hline 	
		7& $r-11$ & $4r+2$ & $(1,9,2r+1,2,r)$ & $\frac{1}{r}(1,9,1)$& $r>11 $, $\modr(r,9)=2$\\ \hline 	
		8& $r-6$ & $4r+2$ & $(2,3,2r+1,2,r)$ & $\frac{1}{r}(2,3,1)$& $r>6 $, $\modr(r,2)\neq 0$, $\modr(r,3)=1$\\ \hline 
		9& $r-8$ & $4r+2$ & $(2,5,2r+1,2,r)$ & $\frac{1}{r}(2,5,1)$& $r>8 $, $\modr(r,2)\neq 0$, $\modr(r,5)=1$\\ \hline 
		10& $r-10$ & $4r+2$ & $(2,7,2r+1,2,r)$ & $\frac{1}{r}(2,7,1)$& $r>10 $, $\modr(r,2)\neq 0$, $\modr(r,7)=3$\\ \hline 
		11& $r-8$ & $4r+2$ & $(3,4,2r+1,2,r)$ & $\frac{1}{r}(3,4,1)$& $r>8 $, $\modr(r,3)\neq 0$, $\modr(r,4)=1$\\ \hline %DONE
		12& $r-9$ & $4r+2$ & $(3,5,2r+1,2,r)$ & $\frac{1}{r}(3,5,1)$& $r>9 $, $\modr(r,3)= 1$, $\modr(r,5)\in \{1,4\}$\\ \hline 
		13& $r-11$ & $4r+2$ & $(3,7,2r+1,2,r)$ & $\frac{1}{r}(3,7,1)$& $r>11 $, $\modr(r,3)\neq 0$, $\modr(r,7)=4$\\ \hline 
		14& $r-10$ & $4r+2$ & $(4,5,2r+1,2,r)$ & $\frac{1}{r}(4,5,1)$& $r>10 $, $\modr(r,5)\in \{1,3\}$\\ \hline 
		15& $r-14$ & $4r+2$ & $(4,9,2r+1,2,r)$ & $\frac{1}{r}(4,9,1)$& $r>14 $, $\modr(r,9)=5$\\ \hline 
		16& $r-12$ & $4r+2$ & $(5,6,2r+1,2,r)$ & $\frac{1}{r}(5,6,1)$& $r>12 $, $\modr(r,5)\in \{1,2\}$, $\modr(r,6)=1$\\ \hline 
		17& $r-13$ & $4r+2$ & $(5,7,2r+1,2,r)$ & $\frac{1}{r}(5,7,1)$& $r>13 $, $\modr(r,5)=1$, $\modr(r,7)=6$\\ \hline 
		18& $r-17$ & $4r+2$ & $(7,9,2r+1,2,r)$ & $\frac{1}{r}(7,9,1)$& $r>17 $, $\modr(r,7)=3$, $\modr(r,9)=8$ \\ 
\end{longtable}
}

\medskip

% Type $X_{4r+4}\subset \bP(a,b,2r+2,4,r)$ with $a,b \leq 5$
{\tiny
\begin{longtable}{|l|l|l|l|l|p{4.5cm}|}
\caption{Type $X_{4r+4}\subset \bP(a,b,2r+2,4,r)$}\label{tab kod 2 4r+4}\\
		\hline
		No.&$\alpha$ & deg & weight & B-weight & conditions \\ \hline
\endfirsthead
\multicolumn{5}{l}{{ {\bf \tablename\ \thetable{}} \textrm{-- continued from previous page}}} \\
\hline 
No.&$\alpha$ & deg & weight & B-weight & conditions \\ \hline
\endhead

 \multicolumn{4}{l}{{\textrm{Continued on next page}}} \\ \hline
\endfoot
\hline 
\endlastfoot

		1& $r-3$ & $4r+4$ & $(1,1,2r+2,4,r)$ & $\frac{1}{r}(1,1,2)$& $r>4 $, $\modr(r,4)=3$\\ \hline 
		2& $r-6$ & $4r+4$ & $(1,3,2r+2,4,r)$ & $\frac{1}{r}(1,3,2)$& $r>6 $, $\modr(r,2)\neq 0$, $\modr(r,3)=2$\\ \hline 
		3& $r-8$ & $4r+4$ & $(1,5,2r+2,4,r)$ & $\frac{1}{r}(1,5,2)$& $r>8 $, $\modr(r,4)=3$, $\modr(r,5)=2$\\ \hline 
		4& $r-10$ & $4r+4$ & $(3,5,2r+2,4,r)$ & $\frac{1}{r}(3,5,2)$& $r>10 $, $\modr(r,2)\neq 0$, $\modr(r,3)\neq 0$, $\modr(r,5)\in \{1,2\}$\\ 

\end{longtable}
}

\medskip

% Type $X_{4r+6}\subset \bP(a,b,2r+3,6,r)$ with $a,b \leq 5$
{\tiny
\begin{longtable}{|l|l|l|l|l|p{4.5cm}|}
\caption{Type $X_{4r+6}\subset \bP(a,b,2r+3,6,r)$}\label{tab kod 2 4r+6}\\
		\hline
		No.&$\alpha$ & deg & weight & B-weight & conditions \\ \hline
\endfirsthead
\multicolumn{5}{l}{{ {\bf \tablename\ \thetable{}} \textrm{-- continued from previous page}}} \\
\hline 
No.&$\alpha$ & deg & weight & B-weight & conditions \\ \hline
\endhead

 \multicolumn{4}{l}{{\textrm{Continued on next page}}} \\ \hline
\endfoot
\hline 
\endlastfoot

		1& $r-5$ & $4r+6$ & $(1,1,2r+3,6,r)$ & $\frac{1}{r}(1,1,3)$& $r>5 $, $\modr(r,6)=4$\\ \hline 
		2& $r-6$ & $4r+6$ & $(1,2,2r+3,6,r)$ & $\frac{1}{r}(1,2,3)$& $r>6 $, $\modr(r,6)=5$\\ \hline 
		3& $r-8$ & $4r+6$ & $(1,4,2r+3,6,r)$ & $\frac{1}{r}(1,4,3)$& $r>8 $, $\modr(r,4)=3$, $\modr(r,6)=1$\\ \hline 
		4& $r-9$ & $4r+6$ & $(1,5,2r+3,6,r)$ & $\frac{1}{r}(1,5,3)$& $r>9 $, $\modr(r,5)=3$, $\modr(r,6)\in\{2,4\}$\\ \hline 
		5& $r-10$ & $4r+6$ & $(2,5,2r+3,6,r)$ & $\frac{1}{r}(2,5,3)$& $r>10 $, $\modr(r,5)\in \{3, 4\}$, $\modr(r,6)=5$\\ \hline 
		6& $r-12$ & $4r+6$ & $(4,5,2r+3,6,r)$ & $\frac{1}{r}(4,5,3)$& $r>12 $, $\modr(r,4)=3$, $\modr(r,5)\in \{2, 3\}$, $\modr(r,6)=1$\\ 
\end{longtable}
}

\subsection{Examples of minimal $3$-folds of general type near the Noether line}\ 

An important topic in studying the geography problem of $3$-folds of general type is the ``Noether inequality in dimension $3$'' which, except for a finite number of families, was proved by Chen, Chen, and Jiang (\cite{Noether, Noether_Add}). The open case is the following:

\begin{conj}\label{Noether} Any minimal projective $3$-fold $X$ of general type with $5\leq p_g\leq 10$ satisfies the inequality
$$K_X^3\geq \frac{4}{3}p_g(X)-\frac{10}{3}.$$
\end{conj}

The effectivity of Theorem~\ref{nefness} makes it possible for us to search those concrete $3$-folds near the Noether line $K^3= \frac{4}{3}p_g-\frac{10}{3}$. We provide here several new examples in Tables~\ref{tableC} and~\ref{tableC+}. Later in Section~\ref{sec 6} we will review these examples in another perspective, from the point of view of the higher dimensional volume problem.

The description of contents of Tables~\ref{tableC} and~\ref{tableC+} are similar to that of Table~\ref{tableA}. Here the last column is the distance $\Delta$ to the Noether line, that is, $\Delta=\Vol(\widehat{X})-\frac{4}{3}p_g(\widehat{X})+\frac{10}{3}$.

 {\tiny
\begin{longtable}{|l|l|l|l|l|l|l|l|l| l |l|}
 			\caption{ Minimal $3$-folds of general type near the Noether line, I}\label{tableC}\\
		\hline
			No.&$\alpha$ & deg & weight & B-weight & $\Vol$&$P_2$ & $p_g$&$\rho$ &basket & $\Delta$ \\ \hline	
\endfirsthead
\multicolumn{5}{l}{{ {\bf \tablename\ \thetable{}} \textrm{-- continued from previous page}}} \\
\hline 
			No.&$\alpha$ & deg & weight & B-weight & $\Vol$&$P_2$ & $p_g$&$\rho$ &basket & $\Delta$ \\ \hline	
\endhead

 \multicolumn{4}{l}{{\textrm{Continued on next page}}} \\ \hline
\endfoot
\hline 
\endlastfoot

1&2&16&$(1,1,2,3,7)$&$\frac{1}{7}(1,1,3)$&$\frac{8}{3}$&11&4&2&$ 2\times(1,3) $& $\frac{2}{3}$\\ \hline
2&3&26&$(1,1,3,5,13)$&$\frac{1}{5}(2,1,1)$&$\frac{7}{2}$&14&5&3&$ (1,2) $ & $\frac{1}{6}$\\ \hline
3&4&36&$(1,1,5,7,18)$&$\frac{1}{7}(2,3,1)$&$\frac{109}{30}$&15&5&2&$ (2,5),\ (1,3),\ (1,2) $ & $\frac{3}{10}$\\ \hline
4&7&56&$(1,2,7,11,28)$&$\frac{1}{11}(4,3,1)$&$\frac{17}{4}$&15&5&9&$ (1,4),\ 2\times(1,2) $& $\frac{11}{12}$\\ \hline
5&2&13&$(1,1,1,3,5)$&$\frac{1}{5}(1,1,1)$&$\frac{16}{3}$&18&6&2&$ (1,3) $& $\frac{2}{3}$\\ \hline
6&2&15&$(1,1,1,3,7)$&$\frac{1}{7}(1,1,3)$&$\frac{16}{3}$&18&6&2&$ (1,3) $& $\frac{2}{3}$\\ \hline
7&5&40&$(1,1,5,8,20)$&$\frac{1}{4}(1,1,1)$&$6$&21&7&6&$ $& $0$\\ \hline
8&6&50&$(1,1,7,10,25)$&$\frac{1}{5}(1,1,2)$&$\frac{85}{14}$&22&7&2&$ (2,7),\ (1,2) $& $\frac{1}{14}$\\ \hline
9&7&56&$(1,1,8,11,28)$&$\frac{1}{11}(2,5,1)$&$\frac{151}{20}$&26&8&2&$ (2,5),\ (1,2),\ (1,4) $& $\frac{13}{60}$\\ \hline
10&9&70&$(1,1,10,14,35)$&$\frac{1}{7}(1,1,3)$&$\frac{301}{30}$&33&10&2&$ (1,2),\ (1,5),\ (1,3) $& $\frac{1}{30}$\\ \hline
11&17&120&$(1,1,17,24,60)$&$\frac{1}{12}(1,1,5)$&22 &65 &19 & 12& &0 \\
\end{longtable}
}

 {\tiny
\begin{longtable}{|l|l|l|l|l|l|l|l|l| l |l|}
 			\caption{ Minimal $3$-folds of general type near the Noether line, II}\label{tableC+}\\
		\hline
		No.&$\alpha$ & deg& weight & B-weight & $\Vol$&$P_2$ & $p_g$ &basket & $\Delta$\\ \hline
\endfirsthead
\multicolumn{10}{l}{{ {\bf \tablename\ \thetable{}} \textrm{-- continued from previous page}}} \\
\hline 
		No.&$\alpha$ & deg& weight & B-weight & $\Vol$&$P_2$ & $p_g$ &basket & $\Delta$ \\ \hline
\endhead

 \multicolumn{10}{l}{{\textrm{Continued on next page}}} \\ \hline
\endfoot
\hline 
\endlastfoot

1&2&15&$(1,1,2,3,6)$&$\frac{1}{6}(1,1,2)$&$\frac{8}{3}$&11&4&$ 2\times(1,3) $ &$\frac{2}{3}$\\ \hline
2&2&17&$(1,1,2,3,8)$&$\frac{1}{8}(1,2,3)$&$\frac{8}{3}$&11&4&$ 2\times(1,3) $&$\frac{2}{3}$\\ \hline
3&2&15&$(1,1,2,2,7)$&$\frac{1}{7}(1,2,2)$&$4$&14&5& &$\frac{2}{3}$\\ 

\end{longtable}
}

%We list the difference between the canonical volume and the Noether bound in Table~\ref{tableC}:
%\begin{longtable}
%{|l|l|l|l|}
% 			\caption{Table $C$-diff}\label{tableCd}\\
%		\hline
%		No.& $\Vol$&Noether bound& difference \\ \hline
% 1 & $\frac{8}{3}$ & $2$ & $\frac{2}{3}$ \\ \hline
% 2 & $\frac{7}{2}$ & $\frac{10}{3}$ & $\frac{1}{6}$ \\ \hline
% 3 & $\frac{109}{30}$ & $\frac{10}{3}$ & $\frac{3}{10}$ \\ %\hline
% 4 & $\frac{17}{4}$ & $\frac{10}{3}$ & $\frac{11}{12}$ \\ \hline
% 5 & $\frac{16}{3}$ & $\frac{14}{3}$ & $\frac{2}{3}$ \\ \hline
% 6 & $\frac{16}{3}$ & $\frac{14}{3}$ & $\frac{2}{3}$ \\ \hline
% 7 & $6$ & $6$ & $0$ \\ \hline
 
 % 8 & $\frac{85}{14}$ & $6$ & $\frac{1}{14}$ \\ \hline
% 9 & $\frac{151}{20}$ & $\frac{22}{3}$ & $\frac{13}{60}$ \\ \hline
% 10 & $\frac{301}{30}$ & $10$ & $\frac{1}{30}$ \\ \hline
%\end{longtable}

\begin{rem}\label{remark noether line}
\begin{enumerate}
\item The minimal $3$-folds $\widehat{X}_{40}$ and $\widehat{X}_{120}$ corresponding to Table~\ref{tableC}, No.~7 and No.~11 are new examples satisfying the Noether equality. In fact, Kobayashi first constructed minimal $3$-folds satisfying the Noether equality in \cite{Kob}, and later Chen and Hu generalizes Kobayashi's construction to get more examples in \cite{C-H}. According to 
\cite[Theorem 1.1]{C-H}, the canonical image of any known example satisfying the Noether equality is a Hirzebruch surface, while the canonical images of $\widehat{X}_{40}$ and $\widehat{X}_{120}$ are $\mathbb{P}(1,1,5)$ and $\mathbb{P}(1,1,17)$ respectively. So both $\widehat{X}_{40}$ and $\widehat{X}_{120}$ are new. 

\item All $3$-folds in Table~\ref{tableC} and Table~\ref{tableC+} are very useful examples in the study of \cite[Question 1.6]{Noether}. Here the $3$-fold corresponding to Table~\ref{tableC}, No.~10, is so far the first singular (non-Gorenstein) minimal $3$-fold which is closest to the Noether line. 

\end{enumerate}
\end{rem}
%%%%% April 30 %%%%%

All items in Table~\ref{tableC} and Table~\ref{tableC+} can be manually verified similar to previous cases.
We illustrate the explicit computations for the last example.

\begin{exmp}[Table~\ref{tableC}, No.~10]\label{ex cj3} Consider the general hypersurface 
$$X=X_{70}\subset \mathbb{P}(1,1,10,14,35)$$
which is verified to be well-formed and quasismooth. It is also clear that $\alpha = 9$, $p_g = 10$, $P_2=33$. The set of singularities of $X$ is
	$$\text{\rm Sing}(X) = \{\frac{1}{2}(1,1,1),\ \frac{1}{5}(1,1,4),\ Q=\frac{1}{7}(1,1,3)\},$$
	where the first two are terminal, and $Q$ is non-canonical.
	
	For the conditions of Theorem~\ref{nefness}, we take 
	$$(b_1,b_2,b_3,b_4,b_5) = (1,1,10,14,35),$$
	$$(e_1,e_2,e_3)=(1,1,3),$$
	$r=7$, and $k=3$. Conditions $(1),(2),(4)$ follow from direct computations (or Remark~\ref{nefrk}). Condition $(3)$ means that the general curve $C_{70}\subset\mathbb{P}(10,14,35)$ should be irreducible. Note that we can not directly use Lemma~\ref{irreducible} here because $\mathbb{P}(10,14,35)$ is not well-formed, but it is easy to see that $\mathbb{P}(10,14,35)\cong 
	\mathbb{P}(5,7,35)\cong\mathbb{P}(1,1,1)$ and $C_{70}$ is isomorphic to a line in the usual projective plane.
	
	So by Theorem~\ref{nefness}, we can take a weighted blow-up $f:\widetilde{X}\to X$ at the point $Q$ with weight $(1,1,3)$ such that $K_{\widetilde{X}}$ is nef. On the exceptional divisor $E$ there is $1$ new singularity of type $\frac{1}{3}(1,1,2),$ which is terminal.
	Hence $\widetilde{X}$ is a minimal $3$-fold. Applying the volume formula for weighted blow-ups (cf. Proposition~\ref{wb}), we get 
	$K_{\widehat{X}}^3=K_{\widetilde{X}}^3 = \frac{301}{30}$.
	Since $\rho(X)=1$ by \cite[Theorem 3.2.4(i)]{WPS}, $\rho(\widehat{X})=\rho(\widetilde{X})=2$. 
	 Finally, we collect the singularities of $\widehat{X}=\widetilde{X}$ and obtain the Reid basket
	$B_{\widehat{X}} = \{ (1,2),\ (1,5),\ (1,3)\}.$
One interesting point of this example is that the Noether distance $\Delta=\frac{1}{30}$, which is the smallest among all known examples not on the Noether line. 
\end{exmp}

 %%%%%%%%%%%%%%%%4.18%%%%%%%%
\section{Canonical volumes of varieties of general type with high canonical dimensions}\label{sec 5}

The motivation of this section is to study Question~\ref{P2}. It is well-known that $v_1=2$ and $v_2=1$. By \cite{EXPI,EXPII,EXPIII} and \cite{Fle00}, we know that $\frac{1}{1680}\leq v_3\leq \frac{1}{420}$. Very little is known about $v_n$ for $n\geq 4$. This also hints that it is very important to find minimal higher dimensional varieties of general type with their canonical volumes as small as possible. 

One may dispart the difficulty of studying $v_n$ by the following strategy. For any $k=1, \cdots, n$, we define
\begin{align*}
v_{n,0}&:=\textrm{min}\{\Vol(Y)|\ Y\ \text{is a}\ n\text{-fold of general type,\ }
p_g(Y)\leq 1\}\ \text{and}\\
v_{n,k}&:=\textrm{min}\{\Vol(Y)|\ Y\ \text{is a}\ n\text{-fold of general type,\ }
\text{can.dim}(Y)=k\}. 
\end{align*}
Clearly, we have $v_n=\text{min}\{v_{n,j}|\ j=0,1,\cdots, n\}$. 

According to Kobayashi \cite{Kob}, we have $v_{n,n}=2$. In this section, we mainly study effective lower bounds for $v_{n,n-1}$ and $v_{n,n-2}$ for $n\geq 4$. 

%or a minimal $n$-fold $Y$ of general type, we mainly study the lower bound of $K_Y^n$ under the assumption that $\textrm{can.dim}(Y)=n-1$ or $n-2$ for $n\geq 4$. 

\subsection{Convention and notation} \ 

\begin{enumerate}
\item
Given any normal projective variety $W$, we say that $W'$ is a higher model of $W$ if $W'$ is normal projective and there is a birational morphism $\sigma: W'\lrw W$. 
\item 
If $W$ is of general type, then $W$ has a minimal model $W_0$ by \cite{BCHM}. Therefore one can find a higher model $W'$ so that there is a birational morphism $\pi_{W'}:W'\lrw W_0$. Sometimes, to avoid too many symbols, we say ``{\it modulo a higher model, there is a birational morphism $\pi_W:W\lrw W_0$}'', which means that we simply replace $W$ by a possibly higher model (but still denoted by $W$). 
\item By convention, {\it an $(a,b)$-surface $S$} means a smooth projective surface of general type whose minimal model $S_0$ has invariants $(K_{S_0}^2, p_g(S_0))=(a,b)$. 
\end{enumerate}

\subsection{The case of canonical dimension $n-1$}\

\begin{thm}\label{v(n-1)} Let $Y$ be a minimal $n$-fold ($n\geq 3$) of general type with $\text{\rm can.dim}(Y)=n-1$. Then %$K_Y^n\geq \frac{2}{n-1}$. 
$$K_Y^n\geq \max\{\frac{2(p_g(Y)-n+1)}{n-1}, \frac{1}{(n-1)^2}\roundup{\frac{8}{3}\big((n-1)(p_g(Y)-n+1)-1\big)}\}.$$
In particular, one has
$$v_{n,n-1}\geq \begin{cases}
\frac{2}{n-1}, &\text{ for }3\leq n\leq 5;\\
\frac{1}{(n-1)^2}\lceil \frac{8(n-2)}{3} \rceil, &\text{ for }n\geq 6.
\end{cases}$$ 
\end{thm}
\begin{proof} Let $\pi_{Y'}: Y'\lrw Y$ be a birational modification such that $|M|=\textrm{Mov}|K_{Y'}|$ is base point free, where $Y'$ is nonsingular and projective. 
\medskip

{\bf Step 0}. Notation and setting. 

Pick up $n-2$ different general members $M_1$, $\dots$, $M_{n-2}\in |M|$. For any $i=1,\dots, n-2$, set $Y^i=M_1\cap\dots\cap M_i$ which, by the Bertini theorem, is a smooth projective subvariety of general type of codimension $i$. We also set $Y^0=Y'$ and $S=Y^{n-2}$. We have the following chain of subvarieties:
$$Y'=Y^0\supset Y^1\supset \dots\supset Y^{n-2}=S.$$

Modulo a higher model of $Y'$, we may and do assume that, for each $i=1,\dots, n-2$, there is a birational morphism $\pi_{Y^i}:Y^i\lrw Y^i_0$ where $Y^i_0$ is a minimal model of $Y^i$. 

We have 
\begin{align}
K_Y^n{}&\geq {\pi^*_{Y'}}(K_Y)^2\cdot (M_1\cdot M_2\cdot\dots \cdot M_{n-2})\notag\\
{}&=({\pi^*_{Y'}}(K_Y)^2\cdot S)=\big(\pi^*_{Y'}(K_Y)|_S\big)^2.\label{OI}
\end{align}
It is clear that $$h^0(Y^i, M|_{Y^i})\geq h^0(Y',M)-i=p_g(Y)-i$$ for all $i>0$. The condition $\text{can.dim}(Y')=n-1$ implies that $|M|_S|$ is a free pencil of curves. Denote by $C$ an irreducible component of a general element of $|M|_S|$. 
Then 
\begin{equation}\pi_Y^*(K_Y)|_S\geq M|_S\equiv \gamma C\label{gamma}\end{equation}
where $\gamma\geq h^0(S, M|_S)-1\geq p_g(Y)-n+1\geq 1$. 

\medskip 

{\bf Step 1}. Canonical restriction inequalities. \ 

For any $i>0$, we have 
\begin{align}
(K_{Y'}+iM)|_{Y^i} & \sim (K_{Y'}+M_1+\dots+M_i)|_{Y^i}\notag\\
& \sim \big( (K_{Y'}+M_1)|_{Y^1}+(M_2+\dots+M_i)|_{Y^1}\big)|_{Y^i}\notag\\
& \sim \big(K_{Y^1}+({M_2}|_{Y^1}+\dots+{M_i}|_{Y^1})\big)|_{Y^i}\notag\\
& \sim \dots \notag\\
&\sim \big(K_{Y^{i-1}}+{M_i}|_{Y^{i-1}}\big)|_{Y^i}\notag\\
&\sim K_{Y^i}. \label{can.div.}
\end{align}

%By Kawamata's extension theorem (cf. \cite[Theorem A]{Kaw98}; see also \cite[Theorem 2.4]{CJ}), 
 
Picking up a sufficiently large and sufficiently divisible integer $m$, we have 
\begin{align*}
|m(1+i)K_{Y'}||_{Y^i}\lsgeq&|m(K_{Y'}+iM)||_{Y^i}\\
=&|m(K_{Y^1}+(i-1)M|_{Y^1})||_{Y^i}\\
 =& \dots\dots \\
=& |m(K_{Y^{i-1}}+M|_{Y^{i-1}})||_{Y^i}\\
=& |mK_{Y^i}|
\end{align*}
for all $i>0$, here we apply \cite[Theorem 2.4(2)]{CJ} to get all above equalities on restrictions of linear systems. 
Since $\text{Mov}|m(1+i)K_{Y'}|=|m(1+i)\pi_{Y'}^*(K_Y)|$ and 
$\text{Mov}|mK_{Y^i}|=|m\pi_{Y^i}^*(K_{Y^i_0})|$, we have 
\begin{equation}
\pi_{Y'}^*(K_Y)|_{Y^i}\geq \frac{1}{i+1}\pi_{Y^i}^*(K_{Y^i_0})\label{cri-1}
\end{equation}
for all $i=1,\dots, n-2$. In particular, 
$$
\pi_Y^*(K_Y)|_S\geq \frac{1}{n-1}\pi_S^*(K_{S_0}).%\label{cri-2}
$$
\noindent{\bf Step 2}. $p_g(S)\geq 3$ and $S$ is not a $(1,2)$-surface. 

By \eqref{gamma} and \eqref{can.div.}, we have
\begin{align}
K_S =&(K_{Y'}+(n-2)M)|_S \notag\\
\geq & (n-1)M|_S\equiv (n-1)\gamma C,\label{can.S}
\end{align}
%where $(n-1)\gamma\geq n-1\geq 3$ as $n\geq 4$. We have 
so
%$$\Vol(S_0)=(\pi_S^*(K_{S_0})\cdot K_S)\geq 3(\pi_S^*(K_{S_0})\cdot C)\geq 3,$$
$$p_g(S)\geq h^0(S, (n-1)M|_S)\geq n\geq 3,$$
which means that $S$ is not a $(1,2)$-surface. Besides, since $(n-1)M|_S$ is free, Relation \eqref{can.S} implies that, at worst numerically, 
\begin{equation}\pi_S^*(K_{S_0})\geq 
(n-1)\gamma C. \label{SS}
\end{equation}

\noindent{\bf Step 3}. The first inequality. 

By \cite[Lemma 2.4]{EXPIII}, we have $(\pi_S^*(K_{S_0})\cdot C)\geq 2$. Hence
\begin{equation}K_Y^n\geq \big(\pi_{Y'}^*(K_Y)|_S\big)^2
 \geq \frac{\gamma}{n-1}(\sigma^*(K_{S_0})\cdot C)\geq \frac{2(p_g(Y)-n+1)}{n-1}. \label{ineq-1}
\end{equation}
Noting that $p_g(Y)\geq n$, we clearly have $K_Y^n\geq \frac{2}{n-1}$. 
\medskip

\noindent{\bf Step 4}. The second inequality. 

By \cite[Proposition 2.9]{Noether} and \eqref{SS}, if $g(C)=2$, we have 
\begin{equation}\Vol(S)\geq \roundup{\frac{8}{3}((n-1)(p_g(Y)-n+1)-1)}. \label{g=2}
\end{equation}
We claim that Inequality \eqref{g=2} also holds for $g(C)\geq 3$. In fact, if we set $C_1={\pi_S}_*(C)$, then Inequality \eqref{SS} reads: $K_{S_0}\geq (n-1)\gamma C_1$. 
When $C_1^2\geq 2$, we clearly have $K_{S_0}^2\geq 2(n-1)^2\gamma^2\geq 4(n-1)\gamma$. 
When $C_1^2\leq 1$, by the adjunction formula, $(K_{S_0}\cdot C_1)+C_1^2=2p_a(C)-2\geq 4$, hence
$(K_{S_0}\cdot C_1)\geq 3$ and
$K_{S_0}^2\geq 3(n-1)\gamma$.
%When $C_1^2=0$, 
%$|(n-1)\gamma C_1|$ is a free pencil and $g(C_1)=g(C)\geq 4$. Then 
%by the adjunction formula, $(K_{S_0}\cdot C_1)=2g(C)-2\geq 4$, hence
%$K_{S_0}^2\geq 4(n-1)\gamma$. 
In a word, both are much better inequalities than \eqref{g=2}. 

Now, by \eqref{OI}, \eqref{SS}, and \eqref{g=2}, we have 
\begin{align}
K_Y^n{}&\geq (\pi_Y^*(K_Y)|_S)^2\geq \frac{1}{(n-1)^2}K_{S_0}^2\notag\\
{}&\geq \frac{1}{(n-1)^2}\roundup{\frac{8}{3}((n-1)(p_g(Y)-n+1)-1)}. \label{ineq-2}
\end{align}
In particular, we have $K_Y^n\geq \frac{1}{(n-1)^2}\roundup{\frac{8}{3}(n-2)}$.

The statement follows automatically from \eqref{ineq-1} and \eqref{ineq-2}. 
\end{proof}

\subsection{The case of canonical dimension $n-2$}\

\begin{prop}\label{3folds} Let $Z$ be a smooth projective $3$-fold of general type such that
$$K_{Z}\sim_{\bQ} l F+D_Z$$
for some positive integer $l$,
where $|F|$ is an irreducible pencil of surfaces with $p_g(F)>0$ such that $|pF|$ is a free pencil for some sufficiently large integer $p$ and $D_Z$ is an effective divisor. Then
$$\Vol(Z)\geq \begin{cases}
\frac{1}{3}, &\textrm{for } l=1;\\
\frac{l^2}{l+1}, &\textrm{for } l\geq 2
\end{cases}$$
with a possible exception when $l=2$ and $F$ is a $(1,1)$-surface. 
\end{prop}
\begin{proof} This is a slight generalization of the main theorem in \cite{Chen07}. For the case $l=1$, the proof of \cite[Theorem 1.4]{Chen07} clearly follows. We need to study the case $l\geq 2$. 

We assume that $f:Z\lrw \Gamma$ is the fibration induced by the free pencil $|pF|$. Then $F$ is a general fiber of $f$. Modulo a higher model of $Z$, we may and do assume that there are birational morphisms $\pi_Z: Z\lrw Z_0$ and $\pi_F:F\lrw F_0$ where $Z_0$ and $F_0$ are minimal models respectively. 

{}First, by applying Kawamata's extension theorem \cite[Theorem~A]{Kaw98} (see \cite[Corollary~2.5]{CZ16}), we have 
$\pi_Z^*(K_{Z_0})|_F\geq \frac{l}{l+1}\pi_F^*(K_{F_0})$. If $K_{F_0}^2\geq 2$, we clearly have
$$
\Vol(Z)\geq l\cdot (\pi_Z^*(K_{Z_0})^2\cdot F)\geq l\cdot \frac{2l^2}{(l+1)^2}>\frac{l^2}{l+1}.
$$

If $(K_{F_0}^2,p_g(F_0))=(1,2)$, we set $\xi_1=\big({\pi_Z^*(K_{Z_0})|_F}\cdot \pi_F^*(K_{F_0})\big)$. Set $|G|=\textrm{Mov}|K_F|=|\pi_F^*(K_{F_0})|$ whose general member is a genus $2$ curve. With $m_0/p=1/l$ and $\beta=\frac{l}{l+1}$, \cite[Theorem 3.1]{CZ08}
implies that for a positive integer $m$, if $(m-2-\frac{2}{l})\xi_1>1$, then 
\begin{align}
m\xi_1\geq 2+\roundup{(m-2-\frac{2}{l})\xi_1}.\label{xi 1}
\end{align}
For a sufficiently large $m$, \eqref{xi 1} implies that $m\xi_1\geq 2+ (m-2-\frac{2}{l})\xi_1$, which implies that $\xi_1\geq \frac{2}{3}$. %Take $m=8$, then $(8-2-\frac{2}{l})\xi_1\geq \frac{10}{3}$ and hence \eqref{xi 1} implies that 
%$8\xi_1\geq 6$. 
Now suppose we proved that $\xi_1\geq \frac{k}{k+1}$ for some integer $k\geq 2$, then $(k+3-2-\frac{2}{l})\xi_1\geq \frac{k^2}{k+1}>k-1\geq 1$, therefore \eqref{xi 1} implies that 
$(k+3)\xi_1\geq 2+k$. So by induction, we eventually get $\xi_1\geq 1$. 

If $l\geq 3$ and $(K_{F_0}^2,p_g(F_0))=(1,1)$, we set 
$$\xi_2=\big({\pi_Z^*(K_{Z_0})|_F}\cdot 2\pi_F^*(K_{F_0})\big). $$ 
Set $|G|=\textrm{Mov}|2K_F|$. The classical surface theory implies that $G\sim 2\pi_F^*(K_{F_0})$ and a general member $C\in |G|$ is an even divisor which is smooth of genus $4$. With $m_0/p=1/l$ and $\beta=\frac{l}{2(l+1)}$, by \cite[Theorem 3.1]{CZ08} and \cite[Lemma 2.2]{EXPIII}, 
for a positive integer $m$, if $(m-3-\frac{3}{l})\xi_2>1$, then 
\begin{align}
m\xi_2\geq 6+2\roundup{\frac{1}{2}(m-3-\frac{3}{l})\xi_2}.\label{xi 2}
\end{align}
For a sufficiently large $m$, \eqref{xi 2} implies that $\xi_2\geq \frac{3}{2}$. Now suppose we proved that $\xi_2\geq \frac{2k+1}{k+1}$ for some integer $k\geq 1$, then $(2k+4-3-\frac{3}{l})\xi_2\geq \frac{2k(2k+1)}{k+1}>2(2k-1)>1$, therefore \eqref{xi 2} implies that 
$(2k+4)\xi_2\geq 6+4k$, that is, $\xi_2\geq \frac{2k+3}{k+2}$. So by induction, we eventually get $\xi_2\geq 2$.

%Taking $m=6,7,\dots$, repeatedly running \cite[Theoem 3.1]{CZ08} and \cite[Lemma 2.2]{EXPIII}, we eventually get $\xi_2\geq 2$. 

Hence in both cases,
$$
\Vol(Z)\geq \frac{l^2}{l+1} ({\pi_Z^*(K_{Z_0})|_F}\cdot \pi_F^*(K_{F_0}))\geq \frac{l^2}{l+1}.
$$
The proof is complete. 
\end{proof}

\begin{lem}\label{1,1} Let $Z$ be a smooth projective $3$-fold of general type such that
$|K_{Z}|\supset |2 F|$
where $|F|$ is a free pencil of surfaces and $F$ is a $(1,1)$-surface. Then $\Vol(Z)\geq \frac{4}{3}$. 
\end{lem}
\begin{proof} This is directly from \cite[3.3, 3.8]{Chen07}. 
\end{proof}

\begin{thm}\label{v(n-2)} Let $Y$ be a minimal $n$-fold of general type with canonical dimension $n-2$ ($n\geq 3$). Then 
$$K_Y^n\geq \begin{cases}
\frac{1}{3}, &\textrm{for } n=3;\\
\frac{(p_g(Y)-n+2)^2}{(n-2)\big( (p_g(Y)-n+2)(n-2)+1 \big)},
&\textrm{for } 4\leq n\leq 11;\\
\frac{2\big(2(n-2)(p_g(Y)-n+2)-3\big)}{3(n-2)^3},
&\textrm{for } n\geq 12.
\end{cases}$$
In particular, one has
$$v_{n,n-2}\geq \begin{cases}
\frac{1}{(n-1)(n-2)}, &\textrm{for } 4\leq n\leq 11;\\
	\frac{4n-14}{3(n-2)^3}, &\textrm{for } n\geq 12.
	\end{cases}$$
\end{thm} 
\begin{proof} We keep the same notation and setting as in Step 0 in the proof of Theorem~\ref{v(n-1)}. By assumption, we have $p_g(Y)\geq n-1$. 

Under the assumption, we only consider $Y^j$ for $j=1,\dots, n-3$. 
Set $X=M_1\cdot\dots\cdot M_{n-3}$, which is a smooth $3$-fold of general type. By assumption, $|M|_X|$ is a free pencil of surfaces on $X$. Take $F$ to be an irreducible component of a general element of $|M|_X|$ and write $M|_X\equiv \mu F$. Since $h^0(X, M|_X)\geq p_g(Y)-n+3\geq 2$, we have 
$$\pi_Y^*(K_Y)|_X\equiv \mu F+D_X$$
where $\mu\geq p_g(Y)-n+2$ and $D_X$ is an effective $\bQ$-divisor on $X$. 
Note also that $M|_X\geq F$ holds as divisors. 

By Relation \eqref{can.div.}, one has
\begin{equation}
K_X\geq (n-2)M|_X\equiv l F \label{pgX}
\end{equation}
where $l= \mu (n-2)\geq (p_g(Y)-n+2)(n-2)$. %checked may 18 night
By Inequality \eqref{cri-1} we have 
$$\pi_Y^*(K_Y)|_X\geq \frac{1}{n-2}\pi_X^*(K_{X_0}).$$

Now, if $n=3$, the statements is directly due to Proposition~\ref{3folds} and Lemma~\ref{1,1} (or just \cite[Theorem 1.4]{Chen07}). 

Assume $n\geq 4$. Then $l\geq (p_g(Y)-n+2)(n-2)\geq 2$. 
Moreover, if $l=2$, then $h^0(X, M|_X)= p_g(Y)-n+3=2$, $\mu= p_g(Y)-n+2=1$ (which imply that $|M|_X|$ is an irreducible free pencil), and $K_X\geq 2M|_X$.
By Proposition~\ref{3folds} and Lemma~\ref{1,1}, we have 
\begin{align*}
K_Y^n\geq &\frac{1}{(n-2)^3}\cdot \Vol(X) \\
=& \frac{(p_g(Y)-n+2)^2}{(n-2)\big( (p_g(Y)-n+2)(n-2)+1 \big)}\\
\geq & \frac{1}{(n-1)(n-2)}.
\end{align*}
On the other hand, whenever $n\geq 12$, we have 
%\begin{equation} 
$$p_g(X)\geq (p_g(Y)-n+2)(n-2)+1\geq 11$$%\label{pg}\end{equation}
by Relation \eqref{pgX}. 
By \cite[Theorem 1.1]{Noether} and \cite[Theorem 1]{Noether_Add}, whenever $n\geq 12$, we have 
\begin{align*}
K_Y^n\geq &\frac{1}{(n-2)^3}\cdot \Vol(X) \\
\geq& \frac{1}{(n-2)^3}(\frac{4}{3}p_g(X)-\frac{10}{3})\\
\geq & \frac{2\big(2(n-2)(p_g(Y)-n+2)-3\big)}{3(n-2)^3}.
\end{align*}
In particular, we have 
$K_Y^n\geq \frac{4n-14}{3(n-2)^3}.$
Combining what we have proved, the statements follow. 
\end{proof}

%%%%%%
\section{Examples attaining minimal volumes}\label{sec 6}

In this section, we provide examples to show that both Theorem~\ref{v(n-1)} and Theorem~\ref{v(n-2)} are optimal or nearly optimal in many cases. 
\smallskip

\noindent {\bf Notation}. For a cyclic quotient singularity $Q=\frac{1}{r}(e_1,e_2,e_3)$, % with $e_1>0$, $e_2>0$, $e_3>0$ and $\gcd(e_1,e_2,e_3)=1$, 
define 
$$\nabla(Q)=\min\big{\{}\sum_{i=1}^3\modr(je_i,r)|\ j=1,\dots, r-1\big{\}}-r.$$
Note that this is independent of the expression of the singularity, i.e., independent of the choice of $(e_1, e_2, e_3)$.
By the ``Canonical Lemma'' (cf. Lemma~\ref{can lem}), $Q$ is canonical (resp. terminal) if and only if $\nabla(Q)\geq 0$ (resp. $>0$). 

\subsection{A construction of higher dimensional varieties from $3$-folds}\ 

%First we have the following:

\begin{prop}\label{HD} Let $X=X^3_d\subset \bP(a_1,a_2,a_3,a_4,a_5)$ be a general well-formed quasismooth hypersurface. % with finitely many non-canonical singularities, say, 
%$$\{Q_i=\frac{1}{r_i}(e_{i,1},e_{i,2},e_{i,3})|\ i=1,\dots,s\}.$$
 % with $s\geq 0$, $e_{i,1}>0$, $e_{i,2}>0$, $e_{i,3}>0$ and $\gcd(e_{i,1},e_{i,2},e_{i,3})=1$ for each $i$. 
Assume that $\alpha=d-\sum_{j=1}^5a_j>1$. Let $Y=Y_d^n\subset \bP(1^{\alpha-1}, a_1,a_2,a_3,a_4,a_5)$ be a general hypersurface, where $n=\alpha+2\geq 4$. 
 %and $\delta_i=r_i-\sum_{j=1}^3e_{i,j}\geq 0$ for each $i=1,\dots,s$. 
 Then
\begin{enumerate}
\item $Y$ is well-formed and quasismooth;
%\item if $X$ has at worst canonical singularities (i.e., $s=0$), then so does $Y$;
%\item if $s>0$ and $\alpha+\nabla(Q_i)\geq 1$ holds for each $i=1,\dots, s$, then $Y$ has at worst canonical singularities; 
\item if $\alpha+\nabla(Q)\geq 1$ (resp. $>1$) holds for every singular point $Q\in X$, then $Y$ has at worst canonical (resp. terminal) singularities; in particular, if $X$ has at worst canonical singularities, then $Y$ has at worst terminal singularities;
\item if $Y$ has at worst canonical singularities, then $p_g(Y)\geq \alpha-1$ and 
$K_Y^n=\frac{d}{a_1a_2a_3a_4a_5}$. 
\end{enumerate}
 \end{prop}
\begin{proof} Item (1) is true by Theorem~\ref{2.7}. Item (3) is just a direct computation as $\mathcal{O}_Y(K_Y)\simeq \mathcal{O}_Y(1)$. 

For Item (2), note that in Proposition~\ref{non iso can}, if $X$ has a cyclic quotient singularity $Q\in X$ of type $\frac{1}{r}(e_{1},e_{2},e_{3})$ (in the form of Proposition~\ref{non iso can}), then $Y$ has a cyclic quotient singularity $\widetilde{Q}$ of type $\frac{1}{r}(1^{\alpha-1}, e_{1},e_{2},e_{3})$. This actually gives a 1-1 correspondence of singularities of $X$ and $Y$. The statement follows from the fact that $\nabla(\widetilde{Q})\geq \alpha-1+\nabla(Q).$
%only need to consider singularities $$\widetilde{Q}_i=\frac{1}{r_i}(1^{\alpha-1}, e_{i,1},e_{i,2},e_{i,3})$$
%for $i=1,\dots, s$. Clearly, for each $i=1,\dots, s$ and for each $j=1,\dots, r_i-1$, 
%$$\sum_{u=1}^3\modr(je_{i,u},r_i)\geq \alpha-1+\nabla(Q_i)+r_i\geq r_i.$$
%Hence by Lemma~\ref{can lem}, each $\widetilde{Q}_i$ is a canonical singularity for $i=1,\dots, s$. 
%For Item (4), since the amplitude of $Y$ is $1$ and $Y$ has at worst canonical singularities, it follows that $K_Y^n=\frac{d}{a_1a_2a_3a_4a_5}$. We are done. 
\end{proof} 

Proposition~\ref{HD} provides us a great amount of concrete examples of higher dimensional minimal varieties of general type, from which we pick up several interesting $n$-folds as follows. 

{\tiny
\begin{longtable}{|l|l|l|}
 			\caption{Examples of $n$-folds}\label{tableX}\\
		\hline
		No.&$3$-fold& resulting $n$-fold \\ 
		\hline
\endfirsthead
\multicolumn{3}{l}{{ {\bf \tablename\ \thetable{}} \textrm{-- continued from previous page}}} \\
\hline 
		No.&$3$-fold&minimal variety $Y_d^n$\\ 
		 \hline
\endhead

 \multicolumn{3}{l}{{\textrm{Continued on next page}}} \\ \hline
\endfoot
\hline 
\endlastfoot

1&$X_{16}^3\subset \bP(1^3, 3,8)$&$Y_{16}^4\subset \bP(1^4, 3,8)$, minimal \\ 
%&&\\
&$\alpha=2$, canonical sing.&$K_Y^4=\frac{2}{3}=\frac{2}{n-1}$, $\text{can.dim}(Y)=3=n-1$\\
\hline
2&$X_{20}^3\subset \bP(1^3, 4,10)$&$Y_{20}^5\subset \bP(1^5, 4,10)$, minimal \\ 
%&&\\
&$\alpha=3$, canonical sing.&$K_Y^5=\frac{1}{2}=\frac{2}{n-1}$, $\text{can.dim}(Y)=4=n-1$\\
\hline
3&$X_{30}^3\subset \bP(1^2, 4,6,15)$&$Y_{30}^5\subset \bP(1^4, 4,6,15)$, minimal\\ 
&$\alpha=3$, canonical sing.&$K_Y^5=\frac{1}{12}=\frac{1}{(n-1)(n-2)}$, $\text{can.dim}(Y)=3=n-2$\\
\hline
4&$X_{70}^3\subset \bP(1^2, 10,14,35)$&$Y_{70}^{11}\subset \bP(1^{10}, 10,14,35)$, minimal\\ 
& (Table~\ref{tableC}, No.~10)& \\
&$\alpha=9$, $\nabla(\frac{1}{7}(1,1,3))=-2$&$K_Y^{11}=\frac{1}{70}$, $\text{can.dim}(Y)=9=n-2$\\
\end{longtable}
}

\begin{rem} Table~\ref{tableX}, No.~1$\sim$3 show that Theorem~\ref{thm(n-1)} is optimal in dimensions $4$ and $5$, and that Theorem~\ref{thm(n-2)} is optimal in dimension $5$. 
\end{rem}

Thanks to the construction in previous sections, we provide more examples which show that Theorem~\ref{v(n-1)} and Theorem~\ref{v(n-2)} are optimal in dimensions $4$ and $5$, and that they are nearly optimal in other dimensions. We have more supporting examples as follows. 

\subsection{Examples of dimension at most $5$}\

\begin{exmp}[{see \cite[Page 151]{Fle00}}] Both Theorem~\ref{v(n-1)} and Theorem~\ref{v(n-2)} are optimal
in the case $n=3$. 
\begin{enumerate}
\item The general hypersurface $3$-fold $X_{12}\subset \bP(1,1,1,2,6)$ has canonical dimension $2=n-1$ and canonical volume $K^3=1$. $X_{12}$ has $2$
terminal cyclic quotient singularities of type $\frac{1}{2}(1,-1,1)$.

\item The general hypersurface $3$-fold $X_{16}\subset \bP(1,1,2,3,8)$ has canonical dimension $1=n-2$ and canonical volume $K^3=\frac{1}{3}$. $X_{16}$ has $3$ terminal cyclic quotient singularities: $2\times \frac{1}{2}(1,-1,1)$, $1\times \frac{1}{3}(1,-1,1)$. 
\end{enumerate}
\end{exmp}

\begin{exmp}[{cf. \cite{GRD}, \cite{B-K}}] The following two $4$-folds of canonical dimension $3=n-1$ attain minimal volumes. 
\begin{enumerate}
%\item[(1)] The general hypersurface 4-fold $Y_{16}\subset \bP(1^4, 3,8)$ is quasismooth and has at worst terminal singularities. The canonical dimension is $3=n-1$ and the canonical volume $K^4=\frac{2}{3}=\frac{2}{n-1}$. On $Y_{16}$, there is only one singularity of type $\frac{1}{3}(1,1,1,2)$. 

\item The general hypersurface 4-fold $Y_{30}\subset \bP(1,2^3, 6,15)$ has amplitude $\alpha=2$, canonical dimension $3=n-1$, and canonical volume $K^4=\frac{2}{3}=\frac{2}{n-1}$. One knows that $Y_{30}$ has at worst canonical singularities: 
 point singularity of type $\frac{1}{3}(1,2,2,2)$ and surface singularities of type 
$\frac{1}{2}(1,1)$. 

\item The general hypersurface 4-fold $Y_{26}\subset \bP(2^4, 3,13)$ has amplitude $\alpha=2$, canonical dimension $3=n-1$, and canonical volume $K^4=\frac{2}{3}=\frac{2}{n-1}$. One knows that $Y_{26}$ has at worst canonical singularities: point singularity of type $\frac{1}{3}(1,2,2,2)$ and surface singularities of type $\frac{1}{2}(1,1)$.
\end{enumerate}
\end{exmp} 

\begin{exmp}[{cf. \cite{GRD}, \cite{B-K}}] The general hypersurface 4-fold $Y_{42} \subset \bP(1,2,2,6,8,21)$ has amplitude $\alpha=2$, canonical dimension $2=n-2$, and canonical volume $K^4=\frac{1}{6}=\frac{1}{(n-1)(n-2)}$. This hypersurface has at worst canonical singularities: point singularities of type $\frac{1}{3}(1,2,2,2)$ and $\frac{1}{8}(1,2,5,6)$, surface singularities of type $\frac{1}{2}(1,1)$. 
\end{exmp}

\subsection{Infinite series of higher dimensional examples}\ 

%In this section we construct examples which are close to the bound predicted in Propositions~\ref{prop1} and~\ref{prop2}. As a by-product, we discovered interesting $3$-folds of general type which is close to the Noether line (cf. \cite{Noether}). 

We start by considering a general hypersurface $n$-fold
 $X_d \subset \mathbb{P}(a_1,...,a_{n+2}).$
 %\subsection{The Case Codimension = n-1}
We present two infinite series of examples for the cases $n = 3k+1$ and $n=3k+2$ where $k$ is a positive integer. %We set the amplitude $\alpha = n-1$.

Referring to Theorem~\ref{v(n-1)}, we set
$$N(n) = \begin{cases}
\frac{8k}{(3k+1)^2}, & \text{ if }n=3k+2;\\
\frac{8k-2}{9k^2}, &\text{ if }n=3k+1.
\end{cases}$$

\begin{exmp}\label{V*} {Varieties $V_{10k+10}$ of dimension $n=3k+2$ and $V_{10k+6}$ of dimension $n=3k+1$ ($k\in \mathbb{N}$). Both have canonical dimensions $n-1$}. 
\begin{enumerate}
\item For all positive integer $k$, setting $n=3k+2$, the general hypersurface $n$-folds
	$$V_{10(k+1)}\subset \mathbb{P}(1^n, 2(k+1),5(k+1))$$
is well-formed quasismooth and has at worst canonical singularities. Since the amplitude $\alpha=1$, 
we see that $p_g(V_{10(k+1)})=n$ and the canonical dimension is $n-1$. {}Finally, the canonical volume 	
	$$\Vol(V_{10(k+1)}) = \frac{1}{k+1}=\frac{3}{n+1}.$$
\item For all positive integer $k$, setting $n=3k+1$, the general hypersurface $n$-folds
	$$V_{10k+6}\subset \mathbb{P}(1^n, 2k+1,5k+3)$$
is well-formed quasismooth and has at worst canonical singularities. Since the amplitude $\alpha=1$, 
we see that $p_g(V_{10k+6})=n$ and the canonical dimension is $n-1$. {}Finally, the canonical volume 	
	$$\Vol(V_{10k+6})= \frac{2}{2k+1}=\frac{6}{2n+1}.$$
\end{enumerate}
\end{exmp}	

\begin{rem}
	\begin{enumerate}
		\item For varieties $V_*$ in Example~\ref{V*}, one has 
		$$\lim_{n\rightarrow \infty} \frac{\Vol(V_*)}{N(n)}= \frac{9}{8},$$	
which means that the lower bound obtained in Theorem~\ref{v(n-1)} is very close to optimum. 			
\item In the proof of Theorem~\ref{v(n-1)}, corresponding to varieties $V_*$ in Example
\ref{V*}, the surface $S$ is a smooth model of either
		$$S_{10(k+1)}\subset \mathbb{P}(1,1,2(k+1),5(k+1)) \ (\alpha=n-1)$$
		or
		$$S_{10k+6}\subset \mathbb{P}(1,1,2k+1,5k+3)\ (\alpha=n-1).$$
	Of course, both $S_{10(k+1)}$ and $S_{10k+6}$ have worse than canonical singularities. 	
	\end{enumerate}
\end{rem}

%%%%% The following part needs to rewrite %%%%%%%

\subsection{More examples}\

We have already seen the power of Proposition~\ref{HD} (cf. Table~\ref{tableX}) in constructing higher dimensional minimal varieties. However, Proposition~\ref{HD} did not tell us what to do if $Y$ is not canonical (i.e., when Proposition~\ref{HD}(2) is not satisfied). A natural idea is to consider the nefness criterion (Theorem~\ref{nefness}), we have the following lemma:

\begin{lem}\label{dim+}
Keep the same notation as in Theorem~\ref{nefness}. Assume that 
the general hypersurface $X=X_d\subset \mathbb{P}(b_1,...,b_{n+2})$, together with the point $Q$ (of type $\frac{1}{r}(e_1, \cdots, e_n)$), satisfies all conditions (1)$\sim$(4) of Theorem~\ref{nefness}.  Here {up to a reordering of $(b_1,...,b_{n+2})$, we always assume that the homogeneous coordinate of $Q$ is $[0:0:...:0:x_{n+1}:x_{n+2}]$.} Set $X'=X'_d\subset \mathbb{P}(1,b_1,...,b_{n+2})$ to be the general hypersurface of degree $d$. Denote by $x_0, x_1,\cdots, x_{n+2}$ the homogeneous coordinates of $\mathbb{P}(1,b_1,...,b_{n+2})$ and set $Q'=(x_0=x_1=\cdots=x_n=0)$. 
Assume further that $e_i = \modr(a_i,r)$ for each $1\leq i\leq n$ and that $\alpha \geq r-\sum^n_{i=1} e_i>1$. 
Then
\begin{itemize}
\item[(a)]
%Assume moreover that 
%$$\alpha \geq r-\sum^n_{i=1} e_i, $$
%then 
$X'=X'_d\subset \mathbb{P}(1,b_1,...,b_{n+2})$
satisfies Theorem~\ref{nefness} $(1)\sim (4)$ as well. %and has amplitude $\alpha'=\alpha-1$.
\item[(b)] %Keep the assumption $\alpha \geq r-\sum^n_{i=1} e_i$, We have: \\
%If $X$ is of general type, so is $X'$.\\
Let $Y'$ be the weighted blow-up of $X'$ at $Q'$ with weight $(1,e_1,\cdots, e_n)$. {Here we can choose this weight thanks to the assumption $e_i = \modr(a_i,r)$.}
If $\nu(Y)=n-1$ and $\alpha = r-\sum^n_{i=1} e_i$, then $\nu(Y') = n$. Otherwise, $\nu(Y') = n+1$. %Here $Y, Y'$ are the blow-ups of $X, X'$ in Theorem~\ref{nefness}.
\end{itemize}
\end{lem}
\begin{proof}
(a) Since we are adding a `$1$' to the weights, $X'$ remains well-formed and quasismooth. Note that $X'$ has amplitude $\alpha'=\alpha-1$.
As $Q=[0:\dots:0:x_{n+1}:x_{n+2}]$ is the unique non-canonical point of $X$, by Lemma~\ref{HD}, the unique (possible) non-canonical point of $X'$ is $Q'=[0:0:\dots:0:x_{n+1}:x_{n+2}]$, which is of type $\frac{1}{r'}(e'_1,...,e'_{n+1})=\frac{1}{r}(1,e_1,...,e_n)$. Since $ r'-\sum^n_{i=0} e'_i= r-\sum^n_{i=1} e_i-1>0$, this is indeed a non-canonical singularity.
We can check that conditions in Theorem~\ref{nefness} hold for $X'$ with $k'$ the $(k+1)$-th place.
% $$(e_1,e_2,e_3) = (\modr(a_1,r), \modr(a_2,r), \modr(a_3,r)).$$
% Since $Q$ can be resolved in one blow-up, so is $Q'$, 
 Conditions $(3)$ and $(4)$ of Theorem~\ref{nefness} are obvious. For conditions (1) and (2), it suffices to note that $\alpha \geq r-\sum_{i=1}^n e_i$ implies
 $$\frac{r'-\sum_{j=1}^{n+1} e'_j}{\alpha'}= \frac{r-\sum_{i=1}^n e_i-1}{\alpha-1} \leq \frac{r-\sum_{i=1}^n e_i}{\alpha}.$$

(b) According to Proposition~\ref{wb}, $\nu(Y') <n+1$ if and only if 
$$\frac{d(\alpha-1)^{n+1}}{\prod_{j=1}^{n+2} b_j}=\frac{(r-\sum_{i=1}^n e_i-1)^{n+1}}{r\prod_{i=1}^n e_i}. $$
On the other hand, $K_Y$ is nef implies that
$$\frac{d\alpha^n}{\prod_{j=1}^{n+2} b_j}\geq \frac{(r-\sum_{i=1}^n e_i)^n}{r\prod_{i=1}^n e_i}, $$
where the equality holds if and only if $\nu(Y) =n-1$.
From the assumption $\alpha\geq r-\sum_{i=1}^n e_i$,
$$1\geq \frac{r-\sum_{i=1}^n e_i}{\alpha}\geq \frac{r-\sum_{i=1}^n e_i-1}{\alpha-1}.$$ 
%which implies that 
%$$\frac{d(\alpha-1)^{n+1}}{\prod b_k}>\frac{(r-\sum_i^n e_i-1)^{n+1}}{r\prod_i^n e_i},$$
%and this is equivalent to that $X'$ is of general type.
So $\nu(Y') =n$ if and only if all above inequalities are equalities, if and only if $\nu(Y) =n-1$ and $\alpha = r -\sum_{i=1}^n e_i$.
%The condition $\nu(X) = n-1$ is equivalent to 
%$$\frac{d\alpha^n}{\prod b_k}=\frac{(r-\sum_i^n e_i)^n}{r\prod_i^n e_i}.$$
%With the further assumption $\alpha = r -\sum_i^n e_i $, we see that above argument still works after replacing all "$>(\geq) $" by "$=$", hence $\nu(X') = n$.
\end{proof}

We apply Lemma~\ref{dim+} to examples in Section~\ref{sec 4}, which results in many higher dimensional examples.

To be more precise, Lemma~\ref{dim+} may be applied to most examples of Table~\ref{tableA} 
% (the condition of Lemma~\ref{dim+} is satisfied for most of the examples in Table~\ref{tableA}), we
to get higher dimensional minimal varieties of general type.
 Lemma~\ref{dim+} works for all examples with $\alpha>1$ in Theorem~\ref{new examples kod 2 summary} (observe that the condition $\alpha = r - \sum e_i$ is satisfied), hence we get higher dimensional minimal varieties $Y$ with $\nu(Y) = \dim Y-1$ (which conjecturally equals to $\kappa(Y)$).

According to the discussion in Section~\ref{sec 5}, we see that the problem of lower bound of canonical volumes for higher dimensional varieties (with large canonical dimensions) is tightly related to the $3$-dimensional Noether inequality. 
Therefore we would like to have a closer second look at Table~\ref{tableC}.

We first apply Proposition~\ref{HD} to Table~\ref{tableC} and find that, except for No.~1, No.~5 and No.~6, all other resulting higher dimensional varieties have at worst canonical singularities, for which we can directly compute their volumes.

For Examples No.~1, No.~5 and No.~6 in Table~\ref{tableC}, one checks that they all satisfy conditions of Lemma~\ref{dim+}. Hence we can compute their volumes as follows. 

\begin{exmp}%{\bf $W^4_{16}$}.
Consider the resulting $4$-fold from Table~\ref{tableC}, No.~1:
	$$W=W_{16}\subset\mathbb{P}(1^3,2,3,7).$$
	It is well-formed and quasismooth with a unique non-canonical singular point $Q=[0:\dots:0:1]$ of type $\frac{1}{7}(1,1,1,3)$.
	By Lemma~\ref{dim+}, conditions of Theorem~\ref{nefness} are satisfied. Hence, after one blow-up at the point $Q$, we get a minimal projective 4-fold $\widetilde{W}$. By Proposition~\ref{wb}, we have $\Vol(\widetilde{W}) = \frac{1}{3}$.	
\end{exmp}

\begin{exmp}%{\bf $W^4_{13}$}.
Consider the resulting $4$-fold from Table~\ref{tableC}, No.~5:
	$$W=W_{13}\subset\mathbb{P}(1^4,3,5).$$
	It is well-formed and quasismooth with a unique non-canonical singular point $Q=[0:\dots:0:1]$ of type $\frac{1}{5}(1,1,1,1)$.
	By Lemma~\ref{dim+}, conditions of Theorem~\ref{nefness} are satisfied. Hence, after one blow-up at the point $Q$, we get a minimal projective 4-fold $\widetilde{W}$. By Proposition~\ref{wb}, we have $\Vol(\widetilde{W}) = \frac{2}{3}$. \end{exmp}

\begin{exmp}%{\bf $W^4_{15}$}.
We consider the resulting $4$-fold from Table~\ref{tableC}, No.~6:
	$$W=W_{15}\subset\mathbb{P}(1^4,3,7).$$
	It is well-formed and quasismooth with a unique non-canonical singular point $Q=[0:\dots:0:1]$ of type $\frac{1}{7}(1,1,1,3)$.
	By Lemma~\ref{dim+}, conditions of Theorem~\ref{nefness} are satisfied. Hence, after one blow-up at the point $Q$, we get a minimal projective 4-fold $\widetilde{W}$. By Proposition~\ref{wb}, we have $\Vol(\widetilde{W}) = \frac{2}{3}$. 
\end{exmp}

	We collect the higher dimensional examples (of amplitude = 1) associated with Table~\ref{tableC}, where for the last column, {\em bound} means the theoretical lower bound (if any) we obtained from Theorem~\ref{v(n-1)} or~\ref{v(n-2)}:
{\small
\begin{longtable}{|l|l|l|l|l|l|l|l|l| p{5cm} |}
\caption{$n$-folds from Table~\ref{tableC}}\label{tableD}\\
		\hline
		No.& dim & deg& weight & can.dim & $\Vol$& bound \\ \hline
		
		1&4&16&$(1^3,2,3,7)$&2&$\frac{1}{3}$&$\frac{1}{6}$\\ \hline
		2&5&26&$(1^4,3,5,13)$&3&$\frac{2}{15}$&$\frac{1}{12}$\\ \hline
		3&6&36&$(1^5,5,7,18)$&4&$\frac{2}{35}$&$\frac{1}{20}$\\ \hline
		4&9&56&$(1^7,2,7,11,28)$&6&$\frac{1}{77}$&\\ \hline
		5&4&13&$(1^4,3,5)$&3&$\frac{2}{3}$&$\frac{2}{3}$\\ \hline
		6&4&15&$(1^4,3,7)$&3&$\frac{2}{3}$&$\frac{2}{3}$\\ \hline
		7&7&40&$(1^6,5,8,20)$&5&$\frac{1}{20}$&$\frac{1}{30}$\\ \hline
		8&8&50&$(1^7,7,10,25)$&6&$\frac{1}{35}$&$\frac{1}{42}$\\ \hline
		9&9&56&$(1^8,8,11,28)$&7&$\frac{1}{44}$&$\frac{1}{56}$\\ \hline
		10&11&70&$(1^{10},10,14,35)$&9&$\frac{1}{70}$&$\frac{1}{90}$\\ \hline
				11&19&120&$(1^{18},17,24,60)$&17&$\frac{1}{204}$&$\frac{62}{3\cdot 17^3}$\\ \hline
\end{longtable}
}
In Table~\ref{tableD},
Examples No.~5 and No.~6 have canonical dimension $n-1$, which are new examples attaining the lower bound of Theorem~\ref{v(n-1)};
Examples No.~1-3 \& 7-11 have canonical dimension $n-2$, offering us new examples with canonical volumes which are very close to the lower bound of Theorem~\ref{v(n-2)};
Example No.~4 has canonical dimension $n-3$, in which case we lack of a reasonable lower bound for the canonical volume. 
 
\appendix
\section{Examples of minimal $3$-folds of Kodaira dimension $2$}\label{appendix}

The description of the content of Table~\ref{tableB} are the same as that of Table~\ref{tableA}.

 { {\tiny
\begin{longtable}{|l|l|l|l|l|l|l|l|l| p{4.5cm} |}
 	\caption{Minimal $3$-folds of Kodaira dimension $2$}\label{tableB}\\
		\hline
			No.&$\alpha$ & deg & weight & B-weight & $\Vol$&$P_2$ & $\chi$&$\rho$ &basket \\ \hline	
\endfirsthead
\multicolumn{5}{l}{{ {\bf \tablename\ \thetable{}} \textrm{-- continued from previous page}}} \\
\hline 
			No.&$\alpha$ & deg & weight & B-weight & $\Vol$&$P_2$ & $\chi$&$\rho$ &basket \\ \hline	
\endhead

 \multicolumn{4}{l}{{\textrm{Continued on next page}}} \\ \hline
\endfoot
\hline 
\endlastfoot

%6r 2r 3r
1.1&1&24&$(1,1,1,8,12)$&$\frac{1}{4}(1,1,1)$&$0$&6&-2&2&$ $\\ \hline
1.2&1&30&$(1,1,2,10,15)$&$\frac{1}{5}(1,1,2)$&$0$&4&-1&2&$ 4\times(1,2) $\\ \hline
1.3&2&42&$(1,1,3,14,21)$&$\frac{1}{7}(1,1,3)$&$0$&7&-2&2&$ 3\times(1,3) $\\ \hline
1.4&3&54&$(1,1,4,18,27)$&$\frac{1}{9}(1,1,4)$&$0$&10&-3&2&$ 2\times(1,4),\ (1,2) $\\ \hline
1.5&4&66&$(1,1,5,22,33)$&$\frac{1}{11}(1,1,5)$&$0$&13&-4&2&$ (2,5),\ (1,5) $\\ \hline
1.6&5&78&$(1,1,6,26,39)$&$\frac{1}{13}(1,1,6)$&$0$&16&-5&2&$ (1,2),\ (1,3),\ (1,6) $\\ \hline
1.7&1&42&$(1,2,3,14,21)$&$\frac{1}{7}(1,2,3)$&$0$&2&0&2&$ 4\times(1,2),\ 3\times(1,3) $\\ \hline
1.8&3&66&$(1,2,5,22,33)$&$\frac{1}{11}(1,2,5)$&$0$&5&-1&2&$ (1,5),\ 4\times(1,2),\ (2,5) $\\ \hline
1.9&3&66&$(1,3,4,22,33)$&$\frac{1}{11}(1,3,4)$&$0$&4&-1&5&$ 2\times(1,4),\ (1,2) $\\ \hline
1.10&2&66&$(1,3,5,22,33)$&$\frac{1}{11}(1,3,5)$&$0$&2&0&2&$ (1,5),\ 3\times(1,3),\ (2,5) $\\ \hline
1.11&3&78&$(2,3,5,26,39)$&$\frac{1}{13}(2,3,5)$&$0$&2&0&5&$ (2,5),\ 4\times(1,2),\ (1,5) $\\ \hline
1.12&1&78&$(3,4,5,26,39)$&$\frac{1}{13}(3,4,5)$&$0$&0&1&2&$ 2\times(1,4),\ (1,5),\ 3\times(1,3),\ (1,2),\ (2,5) $\\ \hline
%15&2&30&$(1,1,1,10,15)$&$\frac{1}{5}(1,1,1)$&$0$&15&-5&2&$ $\\ \hline
%26&3&36&$(1,1,1,12,18)$&$\frac{1}{6}(1,1,1)$&$0$&28&-9&2&$ $\\ \hline
%29&3&42&$(1,1,2,14,21)$&$\frac{1}{7}(1,1,2)$&$0$&16&-5&2&$ 4\times(1,2) $\\ \hline
%31&3&48&$(1,1,3,16,24)$&$\frac{1}{8}(1,1,3)$&$0$&12&-4&5&$ $\\ \hline
%40&4&42&$(1,1,1,14,21)$&$\frac{1}{7}(1,1,1)$&$0$&45&-14&2&$ $\\ \hline
%43&4&78&$(1,3,5,26,39)$&$\frac{1}{13}(1,3,5)$&$0$&5&-1&2&$ (1,5),\ 3\times(1,3),\ (2,5) $\\ \hline
%58&5&54&$(1,1,2,18,27)$&$\frac{1}{9}(1,1,2)$&$0$&36&-11&2&$ 4\times(1,2) $\\ \hline
%61&5&60&$(1,1,3,20,30)$&$\frac{1}{10}(1,1,3)$&$0$&26&-8&2&$ 3\times(1,3) $\\ \hline
%62&5&66&$(1,2,3,22,33)$&$\frac{1}{11}(1,2,3)$&$0$&14&-4&2&$ 4\times(1,2),\ 3\times(1,3) $\\ \hline
%64&5&78&$(1,3,4,26,39)$&$\frac{1}{13}(1,3,4)$&$0$&8&-2&2&$ 2\times(1,4),\ 3\times(1,3),\ (1,2) $\\ \hline
%66&6&66&$(1,1,3,22,33)$&$\frac{1}{11}(1,1,3)$&$0$&35&-11&5&$ $\\ \hline
%78&7&78&$(1,1,4,26,39)$&$\frac{1}{13}(1,1,4)$&$0$&36&-11&2&$ 2\times(1,4),\ (1,2) $\\ \hline
%79&7&78&$(1,2,3,26,39)$&$\frac{1}{13}(1,2,3)$&$0$&24&-7&2&$ 4\times(1,2),\ 3\times(1,3) $\\ \hline
%80&7&96&$(1,3,5,32,48)$&$\frac{1}{16}(1,3,5)$&$0$&11&-3&2&$ (1,5),\ 3\times(1,3),\ (2,5) $\\ \hline
%85&9&96&$(1,1,5,32,48)$&$\frac{1}{16}(1,1,5)$&$0$&46&-14&2&$ (2,5),\ (1,5) $\\ \hline

%3r+3 r r+1
2.1&1&15&$(1,1,3,4,5)$&$\frac{1}{4}(1,1,1)$&$0$&3&-1&2&$ $\\ \hline
2.2&1&18&$(1,2,3,5,6)$&$\frac{1}{5}(1,2,1)$&$0$&2&0&2&$ 4\times(1,2),\ 3\times(1,3) $\\ \hline
2.3&2&24&$(1,3,3,7,8)$&$\frac{1}{7}(1,3,1)$&$0$&3&0&2&$ 9\times(1,3) $\\ \hline
2.4&3&30&$(1,3,4,9,10)$&$\frac{1}{9}(1,4,1)$&$0$&4&-1&5&$ 2\times(1,4),\ (1,2) $\\ \hline
2.5&4&36&$(1,3,5,11,12)$&$\frac{1}{11}(1,5,1)$&$0$&5&-1&2&$ (2,5),\ (1,5),\ 3\times(1,3) $\\ \hline
2.6&5&42&$(1,3,6,13,14)$&$\frac{1}{13}(1,6,1)$&$0$&6&-1&2&$ (1,6),\ 7\times(1,3),\ (1,2) $\\ \hline
2.7&7&51&$(1,3,7,16,17)$&$\frac{1}{16}(1,7,1)$&$0$&9&-3&8&$ $\\ \hline
2.8&1&24&$(2,3,3,7,8)$&$\frac{1}{7}(2,3,1)$&$0$&1&1&2&$ 9\times(1,3),\ 4\times(1,2) $\\ \hline
2.9&3&36&$(2,3,5,11,12)$&$\frac{1}{11}(2,5,1)$&$0$&2&0&5&$ (1,5),\ (2,5),\ 4\times(1,2) $\\ \hline
2.10&5&42&$(3,3,4,13,14)$&$\frac{1}{13}(3,4,1)$&$0$&3&1&2&$ 2\times(1,4),\ 15\times(1,3),\ (1,2) $\\ \hline
2.11&7&51&$(3,3,5,16,17)$&$\frac{1}{16}(3,5,1)$&$0$&4&1&2&$ (1,5),\ (2,5),\ 18\times(1,3) $\\ \hline
2.12&7&54&$(3,4,5,17,18)$&$\frac{1}{17}(4,5,1)$&$0$&3&0&2&$ 2\times(1,4),\ (1,5),\ (2,5),\ 3\times(1,3),\ (1,2) $\\ \hline
%12&2&18&$(1,1,3,5,6)$&$\frac{1}{5}(1,1,1)$&$0$&7&-2&2&$ 3\times(1,3) $\\ \hline
%19&3&21&$(1,1,3,6,7)$&$\frac{1}{6}(1,1,1)$&$0$&13&-4&5&$ $\\ \hline
%20&3&24&$(1,2,3,7,8)$&$\frac{1}{7}(1,2,1)$&$0$&7&-2&2&$ 4\times(1,2) $\\ \hline
%37&4&24&$(1,1,3,7,8)$&$\frac{1}{7}(1,1,1)$&$0$&21&-6&2&$ $\\ \hline
%44&5&30&$(1,2,3,9,10)$&$\frac{1}{9}(1,2,1)$&$0$&16&-4&2&$ 4\times(1,2),\ 3\times(1,3) $\\ \hline
%46&5&33&$(1,3,3,10,11)$&$\frac{1}{10}(1,3,1)$&$0$&11&-2&2&$ 12\times(1,3) $\\ \hline
%51&5&42&$(2,3,5,13,14)$&$\frac{1}{13}(2,5,1)$&$0$&4&-1&6&$ 4\times(1,2) $\\ \hline
%67&7&42&$(1,3,4,13,14)$&$\frac{1}{13}(1,4,1)$&$0$&15&-4&2&$ 2\times(1,4),\ (1,2) $\\ \hline
%68&7&42&$(2,3,3,13,14)$&$\frac{1}{13}(2,3,1)$&$0$&10&-1&2&$ 15\times(1,3),\ 4\times(1,2) $\\ \hline
%81&9&51&$(1,3,5,16,17)$&$\frac{1}{16}(1,5,1)$&$0$&19&-5&2&$ (2,5),\ (1,5) $\\ \hline
%86&10&60&$(3,3,5,19,20)$&$\frac{1}{19}(3,5,1)$&$0$&8&0&10&$ 21\times(1,3) $\\ \hline

%3r+6 r r+2
3.1&5&33&$(1,1,6,9,11)$&$\frac{1}{9}(1,1,2)$&$0$&18&-5&2&$ (1,6),\ (1,2),\ (1,3) $\\ \hline
3.2&5&39&$(1,3,6,11,13)$&$\frac{1}{11}(1,3,2)$&$0$&6&-1&2&$ (1,6),\ (1,2),\ 7\times(1,3) $\\ \hline
3.3&5&45&$(1,5,6,13,15)$&$\frac{1}{13}(1,5,2)$&$0$&4&-1&10&$ (1,6),\ (1,2),\ (1,3) $\\ \hline
3.4&7&57&$(3,5,6,17,19)$&$\frac{1}{17}(3,5,2)$&$0$&2&1&2&$ (2,5),\ (1,6),\ (1,2),\ (1,5),\ 10\times(1,3) $\\ \hline
%4r+2 r 2r+1
4.1&1&18&$(1,1,2,4,9)$&$\frac{1}{4}(1,1,1)$&$0$&4&-1&2&$ 4\times(1,2) $\\ \hline
4.2&1&22&$(1,2,2,5,11)$&$\frac{1}{5}(1,2,1)$&$0$&3&0&2&$ 12\times(1,2) $\\ \hline
4.3&2&30&$(1,2,3,7,15)$&$\frac{1}{7}(1,3,1)$&$0$&4&-1&2&$ 3\times(1,3) $\\ \hline
4.4&3&38&$(1,2,4,9,19)$&$\frac{1}{9}(1,4,1)$&$0$&6&-1&2&$ 2\times(1,4),\ 9\times(1,2) $\\ \hline
4.5&4&46&$(1,2,5,11,23)$&$\frac{1}{11}(1,5,1)$&$0$&7&-2&2&$ (2,5),\ (1,5) $\\ \hline
4.6&5&54&$(1,2,6,13,27)$&$\frac{1}{13}(1,6,1)$&$0$&9&-2&2&$ (1,6),\ 9\times(1,2),\ (1,3) $\\ \hline
4.7&9&82&$(1,2,9,20,41)$&$\frac{1}{20}(1,9,1)$&$0$&16&-5&10&$ 4\times(1,2) $\\ \hline
4.8&1&30&$(2,2,3,7,15)$&$\frac{1}{7}(2,3,1)$&$0$&2&1&2&$ 3\times(1,3),\ 16\times(1,2) $\\ \hline
4.9&3&46&$(2,2,5,11,23)$&$\frac{1}{11}(2,5,1)$&$0$&4&1&2&$ (1,5),\ (2,5),\ 24\times(1,2) $\\ \hline
4.10&7&70&$(2,2,7,17,35)$&$\frac{1}{17}(2,7,1)$&$0$&9&0&11&$ 36\times(1,2) $\\ \hline
4.11&5&54&$(2,3,4,13,27)$&$\frac{1}{13}(3,4,1)$&$0$&5&0&2&$ 2\times(1,4),\ 3\times(1,3),\ 13\times(1,2) $\\ \hline
4.12&7&66&$(2,3,5,16,33)$&$\frac{1}{16}(3,5,1)$&$0$&6&-1&2&$ (1,5),\ (2,5),\ 3\times(1,3),\ 4\times(1,2) $\\ \hline
4.13&3&54&$(2,4,5,13,27)$&$\frac{1}{13}(4,5,1)$&$0$&2&1&2&$ 2\times(1,4),\ (1,5),\ (2,5),\ 13\times(1,2) $\\ \hline
%13&2&22&$(1,1,2,5,11)$&$\frac{1}{5}(1,1,1)$&$0$&9&-3&2&$ $\\ \hline
%21&3&26&$(1,1,2,6,13)$&$\frac{1}{6}(1,1,1)$&$0$&17&-5&2&$ 4\times(1,2) $\\ \hline
%22&3&30&$(1,2,2,7,15)$&$\frac{1}{7}(1,2,1)$&$0$&10&-2&2&$ 16\times(1,2) $\\ \hline
%25&3&34&$(1,2,3,8,17)$&$\frac{1}{8}(1,3,1)$&$0$&7&-2&4&$ 4\times(1,2) $\\ \hline%38&4&30&$(1,1,2,7,15)$&$\frac{1}{7}(1,1,1)$&$0$&27&-8&2&$ $\\ \hline
%47&5&38&$(1,2,2,9,19)$&$\frac{1}{9}(1,2,1)$&$0$&22&-5&2&$ 20\times(1,2) $\\ \hline
%49&5&42&$(1,2,3,10,21)$&$\frac{1}{10}(1,3,1)$&$0$&15&-4&2&$ 3\times(1,3),\ 4\times(1,2) $\\ \hline
%56&5&50&$(1,2,5,12,25)$&$\frac{1}{12}(1,5,1)$&$0$&10&-3&8&$ 4\times(1,2) $\\ \hline
%65&6&46&$(1,2,3,11,23)$&$\frac{1}{11}(1,3,1)$&$0$&20&-6&4&$ $\\ \hline
%71&7&54&$(1,2,4,13,27)$&$\frac{1}{13}(1,4,1)$&$0$&21&-5&2&$ 2\times(1,4),\ 13\times(1,2) $\\ \hline
%72&7&54&$(2,2,3,13,27)$&$\frac{1}{13}(2,3,1)$&$0$&15&-2&2&$ 3\times(1,3),\ 28\times(1,2) $\\ \hline
%82&9&66&$(1,2,5,16,33)$&$\frac{1}{16}(1,5,1)$&$0$&26&-7&2&$ (2,5),\ (1,5),\ 4\times(1,2) $\\ \hline
%83&9&70&$(2,3,4,17,35)$&$\frac{1}{17}(3,4,1)$&$0$&12&-2&4&$ 2\times(1,4),\ 17\times(1,2) $\\ \hline
%87&10&78&$(2,3,5,19,39)$&$\frac{1}{19}(3,5,1)$&$0$&11&-3&6&$ 3\times(1,3) $\\

%4r+4 r 2r+2

5.1&3&32&$(1,1,4,7,16)$&$\frac{1}{7}(1,1,2)$&$0$&10&-3&2&$ (1,2),\ 2\times(1,4) $\\ \hline
5.2&5&48&$(1,3,4,11,24)$&$\frac{1}{11}(1,3,2)$&$0$&8&-2&2&$ (1,2),\ 3\times(1,3),\ 2\times(1,4) $\\ \hline
5.3&1&48&$(3,4,5,11,24)$&$\frac{1}{11}(3,5,2)$&$0$&0&1&2&$ (1,5),\ (1,2),\ (2,5),\ 3\times(1,3),\ 2\times(1,4) $\\ \hline
%77&7&72&$(3,4,5,17,36)$&$\frac{1}{17}(3,5,2)$&$0$&3&0&2&$ (2,5),\ (1,2),\ (1,5),\ 3\times(1,3),\ 2\times(1,4) $\\ \hline

%4r+6 r 2r+3
6.1&5&46&$(1,1,6,10,23)$&$\frac{1}{10}(1,1,3)$&$0$&17&-5&2&$ (1,6),\ (1,3),\ (1,2) $\\ \hline
6.2&5&50&$(1,2,6,11,25)$&$\frac{1}{11}(1,2,3)$&$0$&9&-2&2&$ (1,6),\ (1,3),\ 9\times(1,2) $\\ 
\end{longtable}
}
}

\section{A nefness criterion for blowing up $2$ points}

It is natural to ask whether there exists a nefness criterion when blowing up more than one non-canonical singularities on a weighted hypersurface. Unfortunately, the situation will be more complicated. Here we provide a theorem which handles two points case assuming that the two points are ``strongly'' related to each other. We expect that the next theorem has interesting applications. 
%We can generalize Theorem~\ref{nef pbf} the the case of blowing up two points. We will not use it in this article, but it might also have interesting applications.

\begin{thm}[Nefness criterion II]\label{nef pbf 2}
Let $X\subset \mathbb{P}(b_1, \dots, b_{n+2})$ be an $n$-dimensional well-formed quasismooth general hypersurface of degree $d$ with $\alpha=d-\sum_{i=1}^{n+2}b_i>0$, where $b_1,\dots, b_{n+2}$ are not necessarily ordered by size. 
\begin{itemize}
\item Denote by $x_1,\dots,x_{n+2}$ the homogenous coordinates of $\mathbb{P}(b_1, \dots, b_{n+2})$. Denote by $\ell_1$ and $\ell_2$ the lines $(x_1=x_2=\dots =x_{n}=0)$ and $(x_1=x_2=\dots =x_{n-1}=x_{n+1}=0)$ in $\mathbb{P}(b_1, \dots, b_{n+2})$ respectively. Suppose that $X\cap\ell_1$ and $X\cap \ell_2$ consist of finitely many points and take $Q_1\in X\cap \ell_1$, $Q_2\in X\cap \ell_2$.

\item Assume that $X$ has $2$ cyclic quotient singularities at $Q_1$ of type $\frac{1}{r_1}(e_1,\dots,e_{n-1}, e_n)$ and $Q_2$ of type $\frac{1}{r_2}(f_1,\dots,f_{n-1}, f_{n})$ where $e_i, f_i$ are positive integers, $\gcd(e_1, \dots, e_n)=1$, $\gcd(f_1, \dots, f_n)=1$, $\sum_{i=1}^ne_i<r_1$, $\sum_{i=1}^n f_i<r_2$. 

\item Assume further that $x_1,\dots, x_n$ are the local coordinates of $Q_1$ corresponding to the weights $\frac{e_1}{r_1},\dots,\frac{e_n}{r_1}$ and $x_1,\dots, x_{n-1}, x_{n+1}$ are the local coordinates of $Q_2$ corresponding to the weights $\frac{f_1}{r_2},\dots,\frac{f_n}{r_2}$. 
\end{itemize}

Take $\pi: Y\to X$ to be the weighted blow-up at $Q_1$ and $Q_2$ with weights $(e_1,\dots,e_n)$ and $(f_1,\dots,f_n)$.
Suppose that the following conditions hold: 
\begin{enumerate}
\item $\alpha e_{j} \geq b_{j}(r_1-\sum_{i=1}^ne_i)$ for each $j\in \{1, \dots, n-1\}$;
\item $\alpha f_{j} \geq b_{j}(r_2-\sum_{i=1}^nf_i)$ for each $j\in \{1, \dots, n-1\}$;
\item $\frac{\alpha d}{b_nb_{n+1}b_{n+2}}\geq \frac{r_1-\sum_{i=1}^ne_i}{r_1e_n}+\frac{r_2-\sum_{i=1}^nf_i}{r_2f_n}$;
\item a general hypersurface of degree $d$ in $\mathbb{P}(b_{n}, b_{n+1}, b_{n+2})$ is irreducible;
\item $\mathbb{P}(e_1,\dots,e_n)$ and $\mathbb{P}(f_1,\dots,f_n)$ are well-formed.
\end{enumerate}
Then $K_Y$ is nef and $\nu(Y)\geq n-1$. 
\end{thm}

\begin{proof}
For each $j=1,\dots, n-1$, let $H_j\subset X$ be the effective Weil divisor defined by $x_j=0$ and denote by $L$ a Weil divisor corresponding to $\mathcal{O}_X(1)$. Then $H_j\sim b_j L$ and $K_X\sim \alpha L$. 
Denote by $H'_j$ the strict transform of $H_j$ on $Y$ and by $E_1$, $E_2$ the exceptional divisors of $\pi$ over $Q_1$, $Q_2$. Then 
$$\pi^*H_j=H'_j+\frac{e_j}{r_1}E_1+\frac{f_j}{r_2}E_2.$$ 
Set $t_j=\frac{\alpha e_j-b_j(r_1-\sum_{i=1}^ne_i)}{b_jr_1}$ and $s_j=\frac{\alpha f_j-b_j(r_2-\sum_{i=1}^nf_i)}{b_jr_2}$. Then, for each $j=1,\dots, n-1$, $t_j\geq 0$ and $s_j\geq 0$ by assumption.
As $K_Y=\pi^*K_X-\frac{r_1-\sum_{i=1}^ne_i}{r_1}E_1-\frac{r_2-\sum_{i=1}^nf_i}{r_2}E_2$ and $K_X\sim \frac{\alpha}{b_j}H_j$, 
we can see that 
\begin{align}\label{K=H+E 2}
K_Y\sim_\mathbb{Q} \frac{\alpha}{b_j}H'_j+t_j E_1+s_j E_2
\end{align}
for each $j=1,\dots, n-1$.

Assume, to the contrary, that $K_Y$ is not nef. Then there exists a curve $C$ on $Y$ such that $(K_Y\cdot C)<0$.
Note that $K_Y|_{E_1}=\frac{r_1-\sum_{i=1}^ne_i}{r_1}(-E_1)|_{E_1}$ and $K_Y|_{E_2}=\frac{r_2-\sum_{i=1}^nf_i}{r_2}(-E_2)|_{E_2}$ are ample, hence $C\not \subset E_1\cup E_2$. Therefore Equation \eqref{K=H+E 2} implies that $C\subset \cap_{j=1}^{n-1} H'_j$. 

We claim that $\text{Supp}(\cap_{j=1}^{n-1} H'_j)=C$. It suffices to show that $\text{Supp}(\cap_{j=1}^{n-1} H'_j)$ is an irreducible curve.
Note that $\pi (\text{Supp}(\cap_{j=1}^{n-1} H'_j))=\cap_{j=1}^{n-1} H_j$ is a general hypersurface of degree $d$ in $\mathbb{P}(b_n, b_{n+1}, b_{n+2}),$ hence $\pi (\text{Supp}(\cap_{j=1}^{n-1} H'_j))$ is an irreducible curve by assumption. On the other hand, the support of $\cap_{j=1}^{n-1} H'_j\cap E_1$ is just the point $[0:\dots : 0: 1]$ in $E_1\simeq \mathbb{P}(e_1,\dots, e_n)$ and the support of $\cap_{j=1}^{n-1} H'_j\cap E_2$ is just the point $[0:\dots : 0: 1]$ in $E_2\simeq \mathbb{P}(f_1,\dots, f_n)$. So $\text{Supp}(\cap_{j=1}^{n-1} H'_j)$ is just the strict transform of $\pi (\text{Supp}(\cap_{j=1}^{n-1} H'_j))$, which is an irreducible curve.

Therefore, we can write $(H'_1\cdot \dots \cdot H'_{n-1})=t C$ for some $t>0$ as $1$-cycles. Then $(K_Y\cdot C)<0$ implies that 
$(K_Y\cdot H'_1\cdot \dots \cdot H'_{n-1} )<0$. On the other hand, 
{\begin{align*}{}&(K_Y\cdot H'_1\cdot \dots \cdot H'_{n-1} )\\
={}&{\big(}(\pi^*K_X-\frac{r_1-\sum_{i=1}^ne_i}{r_1}E_1-\frac{r_2-\sum_{i=1}^nf_i}{r_2}E_2)\cdot (\pi^*H_1- \frac{e_1}{r_1}E_1- \frac{f_1}{r_2}E_2)\\
{}&\cdot \cdots \cdot (\pi^*H_{n-1}- \frac{e_{n-1}}{r_1}E_1- \frac{f_{n-1}}{r_2}E_2){\big)}\\
={}&\alpha (\prod_{j=1}^{n-1}b_j ) L^{n}+(-1)^n\frac{(r_1-\sum_{i=1}^ne_i)\prod_{j=1}^{n-1}e_j}{r_1^n}E_1^n\\
{}&+(-1)^n\frac{(r_2-\sum_{i=1}^nf_i)\prod_{j=1}^{n-1}f_j}{r_2^n}E_2^n\\
={}&\frac{\alpha d}{b_nb_{n+1}b_{n+2}}-\frac{r_1-\sum_{i=1}^ne_i}{r_1e_n}-\frac{r_2-\sum_{i=1}^nf_i}{r_2f_n}\geq 0,
\end{align*}}
a contradiction.

Hence we conclude that $K_Y$ is nef. The fact that $\nu(Y)\geq n-1$ follows from $(K_Y^{n-1}\cdot E_1)>0$ as $K_Y|_{E_1}$ is ample.
\end{proof}

\section*{\bf Acknowledgments}

The authors would like to thank J\'anos Koll\'ar and Miles Reid for their useful comments on relevant topics. Part of this paper was written while the authors were enjoying the activity ``Workshop on Algebraic Geometry" at Xiamen University in Nov. 2019. The authors appreciate the hospitality of Xiamen University. The first author is a member of the Key Laboratory of Mathematics for Nonlinear Sciences, Fudan university. We would like to thank the referees for comments and suggestions.

\end{document}